\pdfoutput=1
\documentclass[11pt,reqno]{amsart}
\usepackage{amsfonts}
\usepackage{fullpage}                
\usepackage{graphicx,color}
\usepackage{pinlabel}  
\long\def\red#1{{\color{red}#1}}
\usepackage{amssymb,amsthm}
\usepackage{epstopdf}
\usepackage{framed}
\usepackage{float}  
\usepackage{etoolbox}  
\usepackage{varioref}
\usepackage{chemfig,siunitx}
\usepackage{url}
\usepackage[loose]{subfigure}

\usepackage{etoolbox}

\newcommand{\R}{\mathbb{R}}
\newcommand{\N}{\mathbb{N}}

\let\e\varepsilon
\DeclareMathOperator\supp{supp}
\DeclareMathOperator\diag{diag}
\DeclareMathOperator\sign{sign}
\DeclareMathOperator\range{range}
\DeclareMathOperator\Span{span}
\DeclareMathOperator\rank{rank}
\DeclareMathOperator\rowspace{rowspace}

\newcommand{\ignore}[1]{}
\newtheorem{theorem}{Theorem}[section]
\newtheorem{lemma}[theorem]{Lemma}
\theoremstyle{definition}
\newtheorem{definition}[theorem]{Definition}
\newtheorem{remark}[theorem]{Remark}
\AtEndEnvironment{remark}{\qed}%
\def\sS{\mathcal S}
\def\sR{\mathcal R}
\def\sN{\mathcal N}
\def\sK{\mathcal K}
\def\sm#1{\mathsf{#1}}
\DeclareMathOperator\maxS{maxS}
\DeclareMathOperator\maxP{maxP}
\DeclareMathOperator\maxInvP{maxInvP}

\makeatletter  
\definecolor{MyGray}{rgb}{0.95,0.95,0.95}
\definecolor{shadecolor}{rgb}{0.95,0.95,0.95}
\newdimen\grayboxwidth
\newdimen\grayboxmargin
\newdimen\grayboxinset
\newcounter{example}

\newenvironment{example}%
  {
  \begin{MakeFramed}{\FrameSep1cm\advance\hsize-\width \FrameRestore}}%
  {\end{MakeFramed}}
\makeatother

\usepackage{chemfig}
\usepackage[siunitx]{circuitikz}
\makeatletter
\definearrow3{-|}{%
	\CF@arrow@shift@nodes{#3}%
	\draw[-|](\CF@arrow@start@node)--(\CF@arrow@end@node);%
	\CF@arrow@display@label{#1}{0.5}+\CF@arrow@start@node{#2}{0.5}-\CF@arrow@end@node
}
\definearrow1{s>}{%
\ifx\@empty#1\@empty
\expandafter\draw\expandafter[\CF@arrow@current@style,-CF@full](\CF@arrow@start@node)--(\CF@arrow@end@node);%
\else
\def\curvedarrow@style{shorten <=\CF@arrow@offset,shorten >=\CF@arrow@offset,}%
\CF@expadd@tocs\curvedarrow@style\CF@arrow@current@style \expandafter\draw\expandafter[\curvedarrow@style,-CF@full](\CF@arrow@start@name)..controls#1..(\CF@arrow@end@name);
\fi
}
\definearrow1{s<>}{%
\ifx\@empty#1\@empty
\expandafter\draw[](\CF@arrow@start@node)--(\CF@arrow@end@node);%
\else
\def\curvedarrow@style{shorten <=\CF@arrow@offset,shorten >=\CF@arrow@offset,}%
\CF@expadd@tocs\curvedarrow@style\CF@arrow@current@style \expandafter\draw\expandafter[\curvedarrow@style,CF@full-CF@full](\CF@arrow@start@name)..controls#1..(\CF@arrow@end@name);
\fi
}
\makeatother

\begin{document}

\title{Precision and Sensitivity in Detailed-Balance Reaction Networks}
\author{Tom de Greef \and Saeed Masroor \and Mark A. Peletier \and Rudi Pendavingh}

\begin{abstract}
We study two specific measures of quality of chemical reaction networks, \emph{Precision} and \emph{Sensitivity}. The two measures arise in the study of \emph{sensory adaptation}, in which the reaction network is viewed as an input-output system. Given a step change in input, Sensitivity is a measure of the magnitude of the response,  while Precision is a measure of the degree to which the system returns  to its original output for large time. High values of both are necessary for high-quality adaptation. 

We focus on reaction networks without dissipation, which we interpret as detailed-balance, mass-action networks. We give various upper and lower bounds on the optimal values of Sensitivity and Precision, characterized in terms of the stoichiometry, by using a combination of ideas from matroid theory and differential-equation theory.

Among other results, we show that  this class of non-dissipative systems contains networks with arbitrarily high values of both Sensitivity and Precision. This good performance does come at a cost, however, since certain ratios of concentrations need to be large, the network has to be extensive, or the network should show strongly different time scales.\\

\end{abstract}

\maketitle

\section{Introduction}

\subsection{Dissipation and adaptation}
It has been known at least since Szilard~\cite{Szilard29} and Landauer~\cite{Landauer61} that practical information processing requires the dissipation of free energy. In recent years interest has arisen in the application of this idea to chemical and biochemical systems, and studies have been made of the role of dissipation in decision making~\cite{QianReluga05}, concentration sensing~\cite{Tu08,MehtaSchwab12,SkogeNaqviMeirWingreen13,GovernWolde13TR}, signal transduction~\cite{LapidusHanWang08,BartonSontag13}, behaviour of oscillators~\cite{WangXuWang08,XuZhangWangWang13,CaoWangOuyangTu15}, error correction~\cite{MuruganHuseLeibler12},  sensory adaptation~\cite{LanSartoriNeumannSourjikTu12,LanTu13,SartoriGrangerLeeHorowitz14,Lan15,BoDel-GiudiceCelani15}, and various others.

A common theme in these works is a focus on the  relationship between the quality of the processing on one hand and the magnitude of the dissipation on the other; in most cases the conclusion is that dissipation `improves the situation', in the sense of leading to higher accuracy, speed, or reliability. 
However, much of the current literature is based on single examples, in which simplified models are studied, and often the tuning of the amount of dissipation amounts to a single scalar parameter. 

In fact, given an arbitrary chemical reaction system, there appear to be multiple definitions of the `amount of dissipation' at a given parameter point, which differ in whether they apply only to a stationary point or also to dynamic states, whether they include stochastic fluctuations, and whether they operate at a macroscopic or a microscopic level. This makes it difficult to compare results and make clear statements.

As a first step towards such clear statements, in this paper we approach the problem from the other end: we consider \emph{dissipation-free systems}, and ask the question to which extent these can or can not process information. We do this with the information-processing example of \emph{sensory adaptation}.

\medskip

Adaptation is best explained with the example of the bacterium \emph{E. coli}. As part of its food-finding strategy, \emph{E. coli} propels itself in a straight line through its environment, while monitoring the concentration of e.g.\ glucose outside the cell. Depending on whether this concentration increases or decreases during this `run', it will continue to move for longer or shorter in that direction; upon stopping, it turns to a random new direction, and starts a new run. This stochastic motion has a bias in the direction of increasing concentration, and this is how the bacterium finds its food.

In order to behave in this way, the bacterium has to convert small changes in concentration (a few percent up or down) into large changes in behaviour: ie. it has to show large \emph{Sensitivity}, in the terminology that we define below. At the same time the sensing mechanism has to `reset' or `zero itself' after such a change, in order to be ready for the next change in concentration: a simplified version of this will be called \emph{Precision} below. Most chemical reaction systems do not show this combination of Sensitivity and Precision, but when they do, we call them \emph{adaptive}. 

In this paper we study such adaptive systems, and more precisely, we ask the question 
\begin{quote}
Q. To which extent can a \emph{non-dissipative} system perform such adaptation?
\end{quote}
This question is not only inspired by the general link between dissipation and functionality, already mentioned above. In recent work, Lan, Tu, and co-authors conclude that adaptive systems are necessarily dissipative~\cite{LanSartoriNeumannSourjikTu12,LanTu13}. If that is the case, then the answer to the question above should be `not at all'. 

In fact, the situation turns out to be different, and surprising: we will see that a non-dissipative system is in fact perfectly capable of adaptation,  in the sense that Sensitivity and Precision can both be arbitrarily high. 
However, there are strong limitations, as we shall also see: such good performance requires an extreme setup, in the sense that (a) certain ratios of steady-state concentrations have to be  large, (b) time-scale separation has to be significant, or (c) large networks are required. 

\subsection{Content of this paper}

In order to discuss what `non-dissipative' systems can or can not do, one needs a proper definition of this class of systems. In this paper we define `non-dissipative' systems as \emph{chemical reaction networks with mass-action kinetics satisfying detailed balance}. This is a delicate issue, however, and we discuss it in more detail in Section~\ref{sec:discussion}.

Given this class, question Q above asks for the `most adaptive' behaviour that such a system can perform. To structure this discussion, we will consider the structure of the network---the stoichiometry---to be given, and ask to which extent variation of the kinetic parameters allows the system to have large Precision and Sensitivity (which we define below). 

We start the development in Section~\ref{sec:setup} by defining the various objects and concepts. In Section~\ref{sec:precision} we investigate the Precision, and specifically prove bounds on the \emph{minimal Precision}, or \emph{maximal inverse Precision}. Although we are interested in systems with \emph{large} Precision, not small Precision, the concept of maximal inverse Precision plays a role in the study of Sensitivity in Section~\ref{sec:sensitivity}. There we show that we can make systems with large Precision and Sensitivity, by choosing the stoichiometry and the kinetic parameters in the right way. 

The proof of Theorem~\ref{th:UpperBoundsInvP}, a combinatorial characterization of minimal Precision, is inspired by matroid theory, and in Section~\ref{sec:matroid} we explain this connection. In Section~\ref{sec:dissipative} we give an example of a dissipative system with high Sensitivity and Precision that is small and has only moderate concentration ratios. We conclude with a discussion of the results.

\section{Setup}
\label{sec:setup}

In this section we give mathematical definitions of the objects that we will be considering.

\subsection{Chemical reaction networks}

A chemical reaction network is a set of reactions between chemical species $X_s$, $s\in \sS$,
\[
\sum_{s\in \sS} \alpha_{sr} X_s \leftrightharpoons \sum_{s\in \sS} \beta_{sr} X_s.
\]
Here $\alpha_{sr}$ and $\beta_{sr}$ are the \emph{stoichiometric coefficients}, which we assume to be non-negative numbers. This leads to the following definition.

\begin{definition}[Systems]
\label{def:system}
A \emph{system} is a triple $(\sS,\sR,\sN)$, where
\begin{itemize}
\item $\sS$ is a finite set of \emph{species};
\item $\sR$ is a finite set of \emph{reactions};
\item $\sN= \alpha - \beta$, $\alpha, \beta\geq0$, is a \emph{stoichiometric matrix}, an $\sS\times\sR$ matrix of real numbers such that $\sN_{sr}$ is the relative increase or decrease of species $s$ under reaction $r$.
\end{itemize}
A \emph{kinetic system} is a quadruple $(\sS,\sR,\sN,\sK)$, where $\sK$ is a function that gives, for each set of concentrations $c = (c_s)_{s\in \sS}$ of the species, and for each reaction $r\in \sR$, the net rate of transformation $\sK_r(c)$ in that reaction.
\end{definition}

For any kinetic system, the evolution of the concentrations of the species is given by the ordinary differential equation\footnote{All quantities in this paper can be considered dimensionless, if necessary by non-dimensionalization against standard SI units.}\begin{equation}
\label{eq:ode}
\dot c_s = -\sum_{r\in \sR} \sN_{sr}\sK_r(c)
\qquad \text{or}\qquad
\dot{c}  = - \sN \sK(c).
\end{equation}

\begin{example}
\refstepcounter{example}
\label{ex:A}
\noindent \textbf{Example \ref{ex:A}.}
Consider the following reactions between species $X_1$, $X_2$, and $X_3$:
\begin{equation}
2 X_1  \overset{k_1^+}{\underset{k_1^-}\rightleftharpoons } 
 X_2+X_3,   \qquad 
 X_2   \overset{k_2^+}{\underset{k_2^-}\rightleftharpoons }  X_3.  \nonumber
\end{equation} 
For this system the species set $\sS$ is $\{1,2,3\}$, the reaction set $\sR$ is $\{1,2\}$; the stoichiometric matrix $\sN$ and the kinetic function $\sK$ are
\begin{equation*}
\sN = \left(\begin{matrix}
 2 & 0 \\
 -1 & 1 \\
 -1 & -1
 \end{matrix}\right),
\qquad
\sK(c) =  \left(\begin{matrix}
k_1^+ c_1^2 - k_1^- c_2 c_3  \\
 k_2^+ c_2 - k_2^- c_3
 \end{matrix}\right) =: \begin{pmatrix} \sK_1\\\sK_2 \end{pmatrix}.
\end{equation*}
The dynamics of this reaction network is described by the ODE~\eqref{eq:ode}, which reads for this system
\begin{equation}
\label{eq:ode-exampleA}
\frac{d}{dt} \begin{pmatrix} c_1\\c_2\\c_3 \end{pmatrix} 
= \begin{pmatrix}
  -2\sK_1 \\
  \sK_1 - \sK_2 \\
  \sK_1 + \sK_2
  \end{pmatrix}
= \begin{pmatrix}
  -2k_1^+ c_1^2 +2k_1^- c_2 c_3\\
  k_1^+ c_1^2 - k_1^- c_2 c_3 -  k_2^+ c_2 + k_2^- c_3\\
  k_1^+ c_1^2 - k_1^- c_2 c_3 + k_2^+ c_2 - k_2^- c_3
\end{pmatrix}.
\end{equation}

\begin{figure}[H]
\labellist
\small
\pinlabel{time} [t] at 245 2
\pinlabel{\rotatebox{90}{concentration}} [r] at 2 180
\endlabellist
\includegraphics[height=0.3\textwidth]{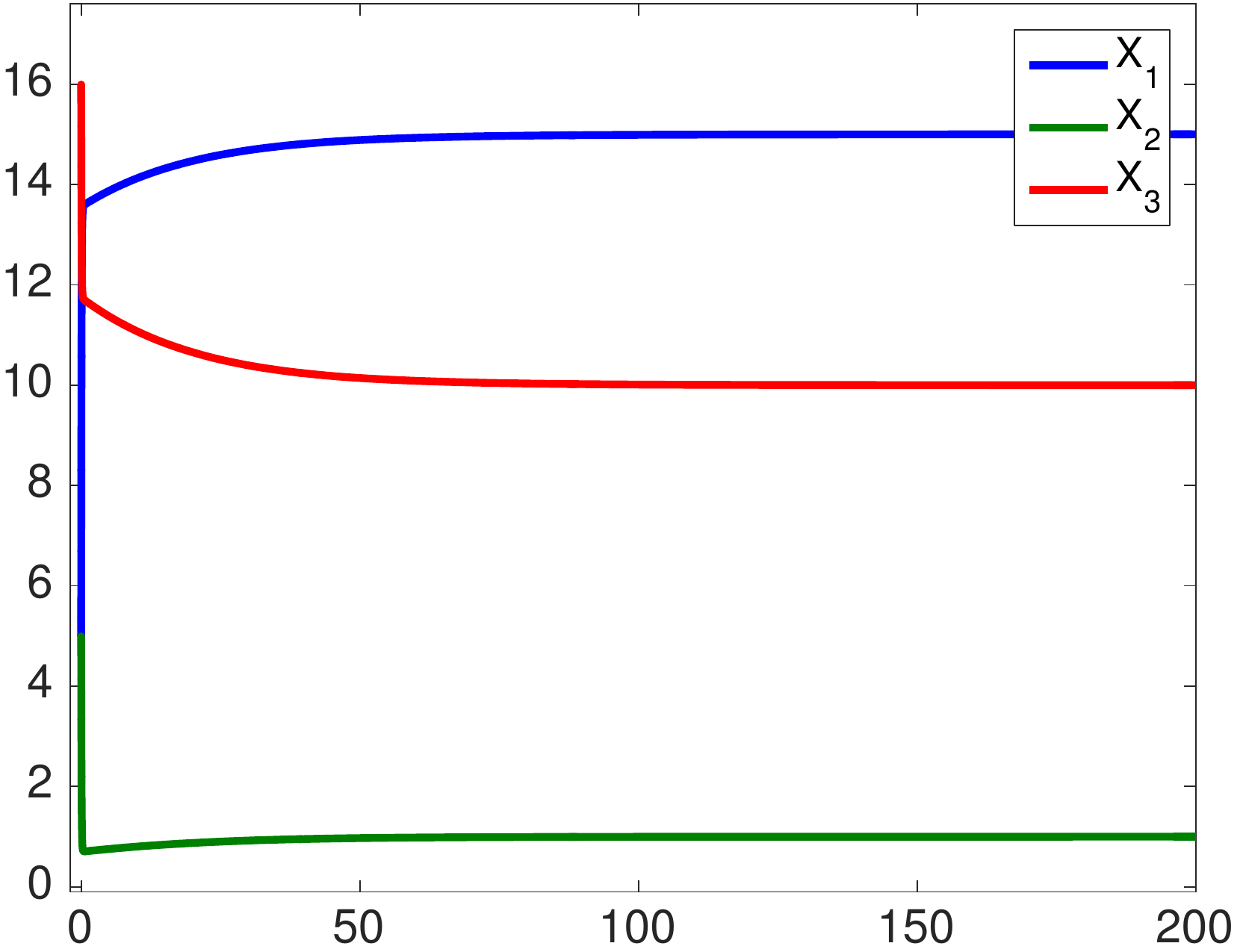}
\qquad 
\labellist
\small
\pinlabel{time} [t] at 245 10
\pinlabel{\rotatebox{90}{concentration}} [r] at 2 180
\endlabellist
\includegraphics[height=0.3\textwidth]{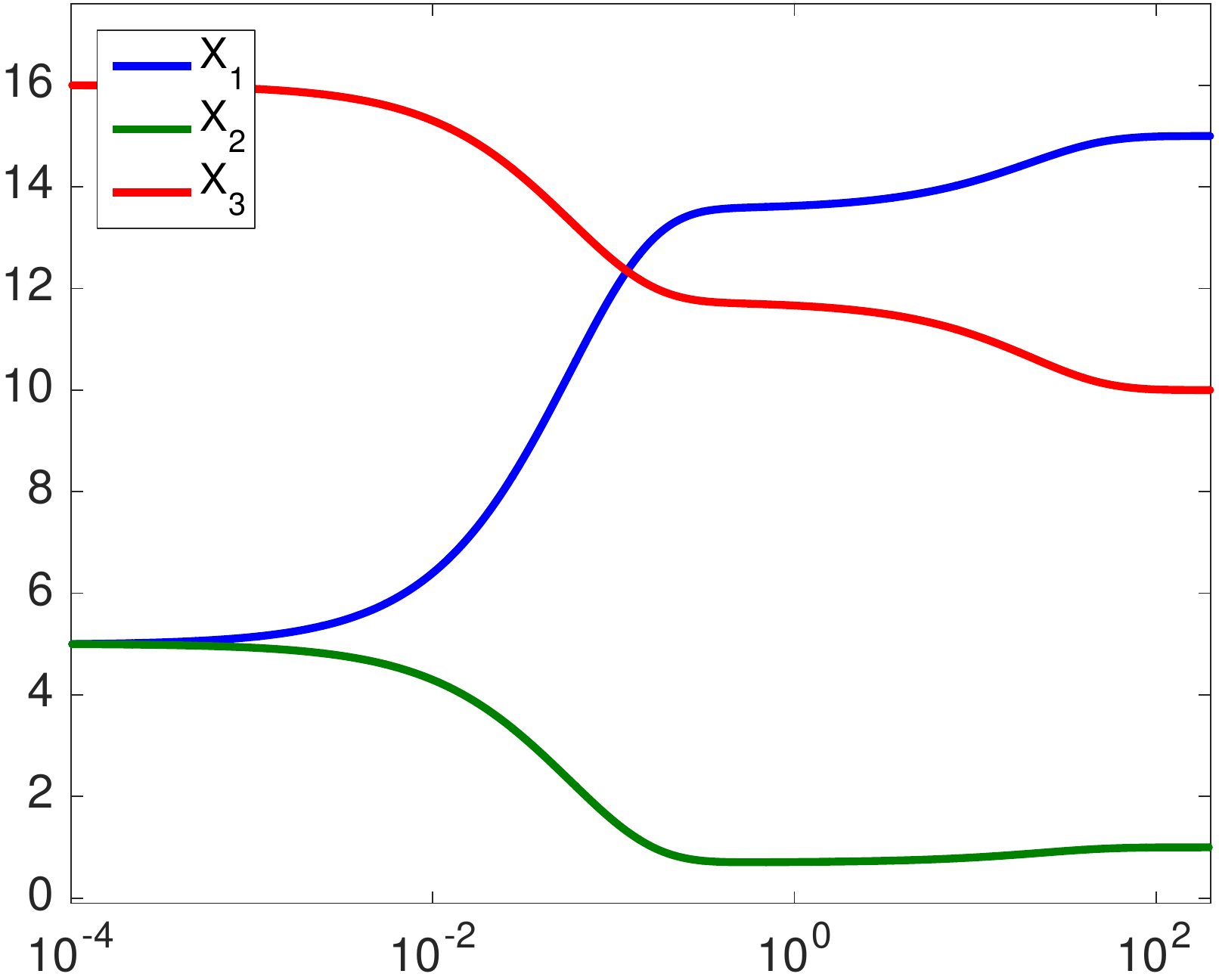}
\vskip6\jot
\caption{A solution of equation~\eqref{eq:ode-exampleA}, with parameters 
$k^+_1 = 0.044$, 
$k^-_1 = 1$,
$k^+_2 = 0.1$,
$k^-_2 = 0.01$, and
$c(0)= (5,5,16)^T$ . The system reaches a steady state $\overline{c}=(15,10,1)^T$. The same solution is plotted both against linear time (left) and logarithmic time (right).  The evolution contains multiple time scales; in the right figure the behaviour at these different time scales is easier to recognize, and for this reason we mostly use logarithmic time axes below.}
\label{fig:exampleA}
\end{figure}
\end{example}

\subsection{Mass-action and detailed balance}

\begin{definition}[Mass-action, detailed-balance kinetics]
Given a system $(\sS,\sR,\sN,\sK)$, \emph{mass-action detailed-balance kinetics} is given by functions $\sK_r$ of the form\footnote{Mass-action kinetics are of the form
\begin{equation*}
\sK_r(c) =  k_r^+ c^{\alpha_r} - k_r^- c^{\beta_r} 
\end{equation*}
and the structure~\eqref{eq:DBMA-kinetics} follows from the detailed-balance assumption that there exists a stationary state $\overline c$ at which all reactions are in equilibrium. See e.g.~\cite{ErdiToth89}.}
\begin{equation}
\label{eq:DBMA-kinetics}
\sK_r(c) = k_r\biggl[ \Bigl(\frac{c}{\overline c}\Bigr)^{\alpha_{r}} - \Bigl(\frac{c}{\overline c}\Bigr)^{\beta_{r}}\biggr],
\end{equation}
where we write $\alpha_r,\beta_r$ for the column vector of $\alpha$ and $\beta$ corresponding to reaction $r$, and the notation $c^{\alpha_r}$ stands for the monomial $\prod_{s\in \sS} c_s^{\alpha_{sr}}$. The function~\eqref{eq:DBMA-kinetics} is characterized by the parameters $(\overline c,k)\in \R^\sS_+\times \R^\sR_+$.
\end{definition}
We often write $(\sS,\sR,\sN,(\overline c,k))$\footnote{This 4-tuple does not completely determine $\alpha$ and $\beta$, only $\sN=\alpha-\beta$, and therefore does not contain enough information to characterize the full kinetics~\eqref{eq:DBMA-kinetics}. The linearization of~\eqref{eq:DBMA-kinetics} at a stationary point only depends on $\sN=\alpha-\beta$, however, and since the Sensitivity and Precision are defined in terms of this linearization, this contains the necessary information for our purposes.} for the kinetic system generated by this pair $(\overline c,k)$. Note that $c = \overline c$ is a stationary point for~\eqref{eq:DBMA-kinetics}, but there typically are other stationary points. This is related to the fact that  when $\range(\sN)$ is a strict subspace of $\R^\sS$, then the evolution~\eqref{eq:ode} takes place in a subset, a simplex:
\begin{definition}[Stoichiometric simplex]
Let $W = \range(\sN)$. For any $\gamma\in \R_+^\sS$, the \emph{stoichiometric simplex} is the relatively open simplex
\[
G(\gamma) := \bigl(\gamma + W\bigr) \cap \R_+^\sS.
\]
\end{definition}
The stoichiometric simplex is the set of positive concentrations that can be reached by starting from $\gamma$ and assigning arbitrary rates to each of the reactions. It is invariant under the evolution~\eqref{eq:ode}.

\begin{example}
\noindent
\textbf{Example \ref{ex:A} (continued).} 
For Example \ref{ex:A}, the range of $\sN$ is the set $(1,1,1)^\perp$, implying that the ODE~\eqref{eq:ode-exampleA} admits the conservation law $c_1+c_2+c_3 = \mathrm{constant}$. Consequently, the stoichiometric simplices are the sets
\[
\{\,(c_1,c_2,c_3)\in \R^3_+: c_1+c_2+c_3 = \mathrm{constant}\, \}.
\]
\end{example}

\begin{lemma}[Stationary states]
\label{lem:EL}
For given $(\overline c,k)\in \R^\sS_+\times\R^\sR_+$, there exists exactly one stationary point of~\eqref{eq:ode} in each simplex $G(\gamma)$, and each solution in $G(\gamma)$ converges to it for large time. We indicate this stationary point by $\hat c[\gamma]$.
The mapping $\gamma \mapsto \hat c[\gamma]$ is a smooth mapping from $\R_+^\sS$ to $\R_+^\sS$; $\hat c$ is the unique solution of the equations
\begin{equation}
\label{eq:EL}
\sN^T \log \frac {\hat c[\gamma]}{\overline c} = 0, \qquad \hat c[\gamma]\in G(\gamma).
\end{equation}
\end{lemma}
\begin{proof}
The existence, uniqueness, and attraction properties are well-known in the field of chemical reaction theory; see e.g.~\cite[Th.~6A]{HornJackson72}. The equation~\eqref{eq:EL} can be verified by inspecting~\eqref{eq:DBMA-kinetics} or by using the fact that free energy decreases along a solution~\cite{HornJackson72,MaasMielkeInPrep}. Finally, the smooth dependence follows from applying the implicit function theorem  to~\eqref{eq:EL}.
\end{proof}

\begin{remark}[Modifying $\overline c$]
Below we will consider systems described by parameters $(\overline c,k)$, and concentrate on linearizations at some stationary point $\hat c[\gamma]$, which \emph{a priori} need not be equal to $\overline c$. However, without loss of generality we can assume that the stationary point equals $\overline c$, by describing the same system by a new but equivalent set of parameters $(\hat c[\gamma],\widetilde k)$, where
\[
\widetilde k_r 
:= k_r \Bigl(\frac{\hat c[\gamma]}{\overline c}\Bigr)^{\sN_r^+} 
= k_r \Bigl(\frac{\hat c[\gamma]}{\overline c}\Bigr)^{\sN_r^-} .
\]
The equations~\eqref{eq:ode}-\eqref{eq:DBMA-kinetics} are identical for $(\overline c,k)$ and $(\hat c[\gamma],\widetilde k)$.
We will therefore always assume that the stationary point under consideration is $\overline c$.
\end{remark}

\begin{remark}[Independence of $k$]
As long as $k$ is a vector with strictly positive components, $\hat c[\gamma]$ is independent of $k$, as can be recognized from the absence of $k$ from~\eqref{eq:EL}. (If one of the components of $k$ vanishes, however, this amounts to removing a column from $\sN$, which modifies~\eqref{eq:EL} and leads to a different equation. We will use this idea below.)
\end{remark}

\subsection{Precision and Sensitivity for detailed-balance reaction networks}
\label{subsec:Precison-Sensitivity-definition}
%

We now think of a chemical reaction network as an input-output system, and we restrict ourselves to detailed-balance, mass-action kinetics networks. The input and output variables are concentrations, indicated by $i,o\in \sS$, $i\not= o$. 

The rest of this paper is based on the following setup.

\medskip
\begin{quote}
\textbf{The adaptation experiment.} Prepare the system in a steady state $\overline c$; at time zero, instantaneously add an amount of the input species to the system, thus increasing~$c_i$; observe the evolution of the output variable $c_o$ (see Figure~\ref{fig:informalAdaptation}).
\end{quote}

\medskip

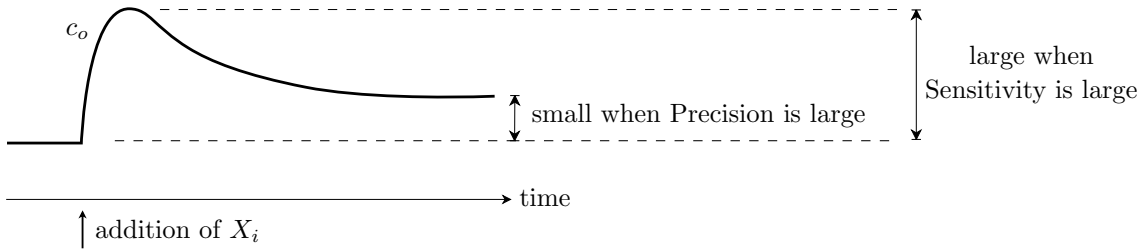
\begin{figure}[ht]
\labellist
\small
\pinlabel {addition of $X_i$} [l] at 44 10
\pinlabel time  by -1.1 0 at 270 28
\pinlabel {small when Precision is large} [l] at 278 72
\pinlabel {$\substack{\text{\small large when}\\[\jot]\text{\small Sensitivity is large}}$} [l] at 489 94
\pinlabel $c_o$ [br] at 48 110
\endlabellist
\includegraphics[width=0.9\textwidth]{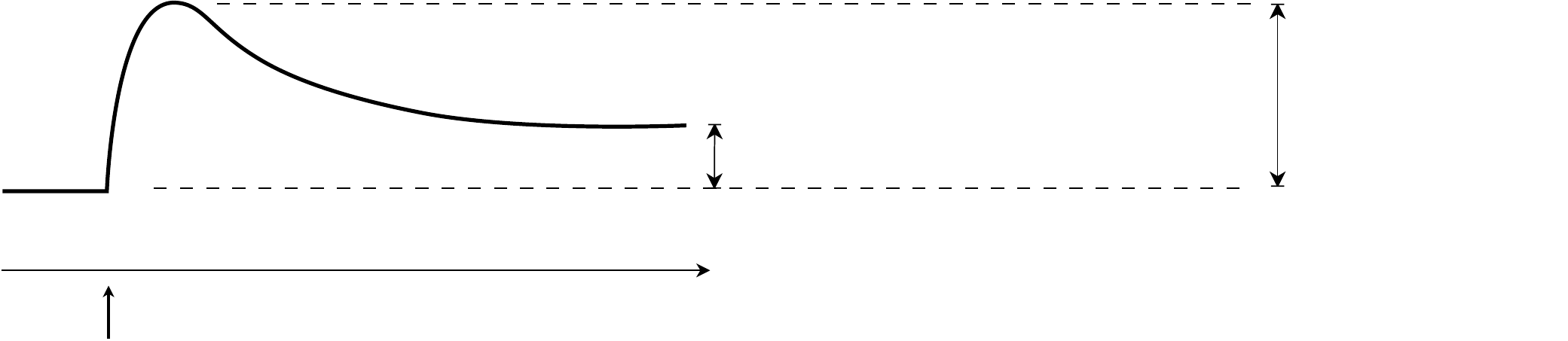}
\caption{The adaptation experiment. The system starts in steady state, and the output $c_o$ is constant. At time zero a small amount of species $X_i$ is added, leading to a small increase in concentration $c_i$. The system is no longer in steady state, and it evolves towards a new steady state. \emph{Sensitvity} is a measure of the maximal deviation along the time course, and \emph{Precision} is a measure of the degree to which the new steady state is close to the previous one.}
\label{fig:informalAdaptation}
\end{figure}

In this situation, the \emph{Sensitivity} is defined\footnote{%
This terminology follows~\cite{MaTrusinaEl-SamadLimTang09}. Note however that the term `sensitivity' also may refer to the variation of a stationary state under variation of a parameter, as in `parameter sensitivity'~\cite{HeinrichSchuster96} or more generally as the depedence of a model prediction on the assumptions and parameters~\cite{SaltelliRattoAndresCampolongoCariboniGatelliSaisanaTarantola08}.}  as a normalized measure of the strength of the response of $c_o$ to the change in $c_i$:

\begin{definition}[Sensitivity]
\label{def:sensitivity}
Given a detailed-balance, mass-action system $(\sS,\sR,\sN,(\overline c,k))$,  and given a choice of input and output species $i,o\in\sS$, 
the Sensitivity is defined as
\begin{equation}
\label{defeq:sensitivity}
S := \lim_{\e\to0} \frac{\ \log \,\sup\limits_{t\geq0} c^\e_o(t) - \log \overline c_o\ }{\log c^\e_i(0) - \log \overline c_i\strut}
=  \lim_{\e\to0} \frac{\ \log \sup\limits_{t\geq0} c^\e_o(t) - \log c_o^\e(0)\ }{\log c^\e_i(0) - \log c^0_i},
\end{equation}
where  $c^\e(t)$ is the solution of~\eqref{eq:ode} with initial datum $c^\e(0)={\overline c+ \e e_i}$.
\end{definition}

This could also be written in shorthand notation as 
\[
S = \frac {d\log \sup_{t\geq0} c_o(t)}{d\log c_i(0)}.
\]

High Sensitivity indicates that small increases in input concentration $c_i$ lead to large swings in output~$c_o$. The appearance of the logarithms both for $c_i$ and for $c_o$ means that \emph{relative} changes are measured. This is related to the fact that  mass-action kinetics makes the response to absolute changes dependent on the reference value. (There is recent interest in networks providing exact \emph{fold-change} responses, which are sensitive to \emph{relative} changes, but otherwise independent of the reference value (e.g.~\cite{GoentoroShovalKirschnerAlon09}). This corresponds to the Sensitivity above being independent of the parameter point~$\overline c$ at which it is measured.) Logarithmic derivatives are also used in Metabolic Control Analysis~\cite{HeinrichSchuster96,Fell97}.

\medskip

The \emph{Precision}, on the other hand, refers to the degree to which the output settles back to the original value at long times:
\begin{definition}[Precision, \cite{MaTrusinaEl-SamadLimTang09}]
\label{def:precision}
In the same context as Definition~\ref{def:sensitivity}, the Precision $P$ is defined through its inverse,
\begin{equation}
\label{defeq:precision} 
P^{-1} := \lim_{\e\to0} \frac{\log c^\e_o(+\infty) - \log \overline c_o}{\log c^\e_i(0) - \log \overline c_i\strut}
=  \lim_{\e\to0} \frac{\ \log c^\e_o(+\infty) - \log c_o^\e(0)\ }{\log c^\e_i(0) - \log c^0_i}.
\end{equation}
\end{definition}

High Precision indicates that the stationary output changes little when the parameter point changes---again, both measured in relative magnitudes. 

\medskip

Since both Precision and Sensitivity are defined in terms of small-perturbation limits, they have equivalent definitions in terms of a linearized version of equation~\eqref{eq:ode}. Because of the logarithmic derivatives, the most convenient form of this equation arises by perturbing the stationary state~$\overline c$ {multiplicatively}: if we set $c^\e(t)=(1+\e u(t))\overline{c}$, then to leading order the function $u$ solves the equation
\begin{equation}
\label{eq:ode-linearized}
\dot u = A u, \qquad
A_{ss'} =  -\frac{1}{\overline c_{s}}\sum_{r\in \sR} \sN_{sr}k_r  \sN_{s'r}.
\end{equation}
(This equation can also be found by linearizing~\eqref{eq:ode} in the usual way, and transforming to new coordinates, scaled by $\overline c$). For given $u(0)$, the solution of this equation is $u(t) = e^{tA}u(0)$. 

\begin{lemma}[Alternative formulations of Precision and Sensitivity]
Again in the same context, let $t\mapsto u(t)$ be the solution of~\eqref{eq:ode-linearized} with initial data $u(0)=e_i$. 
The Precision and Sensitivity then have the alternative formulations
\begin{equation}
\label{eq:alt-def-precision-and-sensitivity}
P^{-1} = \lim_{t\to\infty} u_o(t) = \lim_{t\to\infty} (e^{tA})_{oi}
\qquad
\text{and}
\qquad
S =  \sup_{t\geq0} u_o(t) = \sup_{t\geq0} (e^{tA})_{oi}.
\end{equation}
In addition, recalling the notation $\hat c[\gamma]$ for the stationary state in the stoichiometric simplex containing $\gamma$, we have
\begin{equation}
\label{altdef:P}
P^{-1} = \frac{\overline c_i}{\overline c_o} 
  \lim_{\e\downarrow 0} \frac{1}\e \bigl( \hat c[\overline c+\e e_i] - \overline c\bigr)_o.
\end{equation}
\end{lemma}

\begin{proof}
These formulas follow by direct manipulation.
\end{proof}

\begin{example}
\noindent
\textbf{Example \ref{ex:A}  (continued).}
First we demonstrate an adaptation experiment for a small but finite $\e$. Fix $(\overline{c},k)$ in~\eqref{eq:ode-exampleA}. We perturb the system by adding $\e$ of $X_1$ to the system. The evolution is shown in Figure~\ref{fig:exA-c}. Next, in the limit of small perturbations, the matrix $A$ in~\eqref{eq:ode-linearized} is found to be
 \begin{equation*}
 A = \left(
  \begin{array}{ccc}
    -\dfrac{4 k_1}{\overline{c}_1}\phantom- 
      & \phantom-\dfrac{2 k_1}{\overline{c}_1}
      & \phantom-\dfrac{2 k_1}{\overline{c}_1} \\[4\jot]
    \phantom-\dfrac{2 k_1}{\overline{c}_2}\phantom- 
      & -\dfrac{k_1+k_2}{\overline{c}_2} 
      & \phantom-\dfrac{k_2-k_1}{\overline{c}_2} \\[4\jot]
    \phantom-\dfrac{2 k_1}{\overline{c}_3}\phantom-
      & \phantom-\dfrac{k_2-k_1}{\overline{c}_3}
      & -\dfrac{k_1+k_2}{\overline{c}_3}  \\
  \end{array}
  \right).
 \end{equation*}
The expressions~\eqref{eq:alt-def-precision-and-sensitivity} for Precision and Sensitivity imply that the we can obtain these two quantities by plotting the time trajectory of the corresponding entry in the matrix exponential $(e^{tA})_{oi}$. In this example we take $X_1$ to be the input and we plot the three entries in the first column of $(e^{tA})$ in Figure~\ref{fig:exA-u}.
\begin{figure}[H]
\centering
\labellist
\small
\pinlabel{time} [t] at 245 2
\pinlabel{\rotatebox{90}{conc.}} [r] at -7 90
\pinlabel{\rotatebox{90}{conc.}} [r] at -7 220
\pinlabel{\rotatebox{90}{conc.}} [r] at -7 340
\endlabellist
\includegraphics[height=0.3\textwidth]{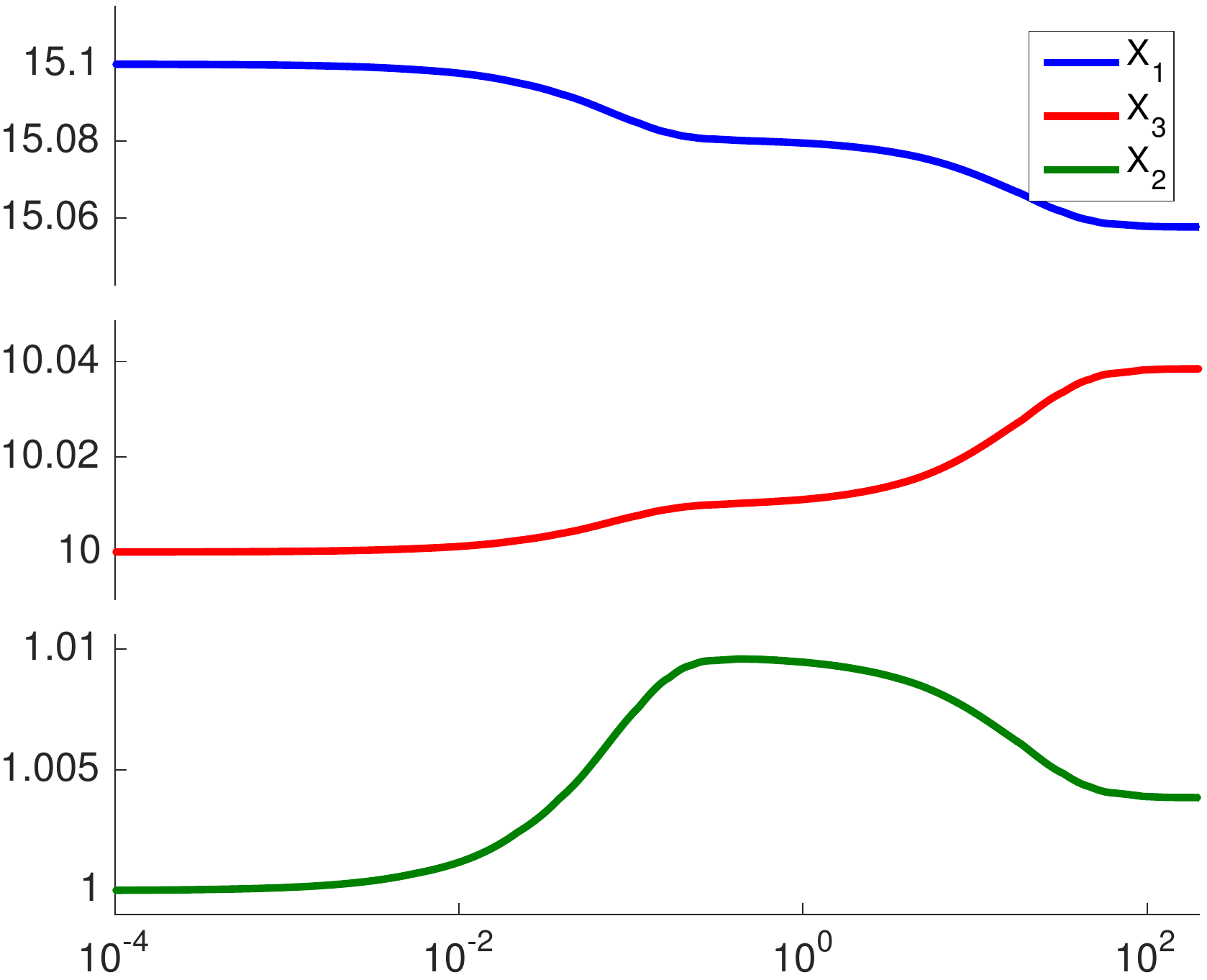}
\qquad 
\labellist
\small
\pinlabel{time} [t] at 265 0
\endlabellist
\includegraphics[height=0.3\textwidth]{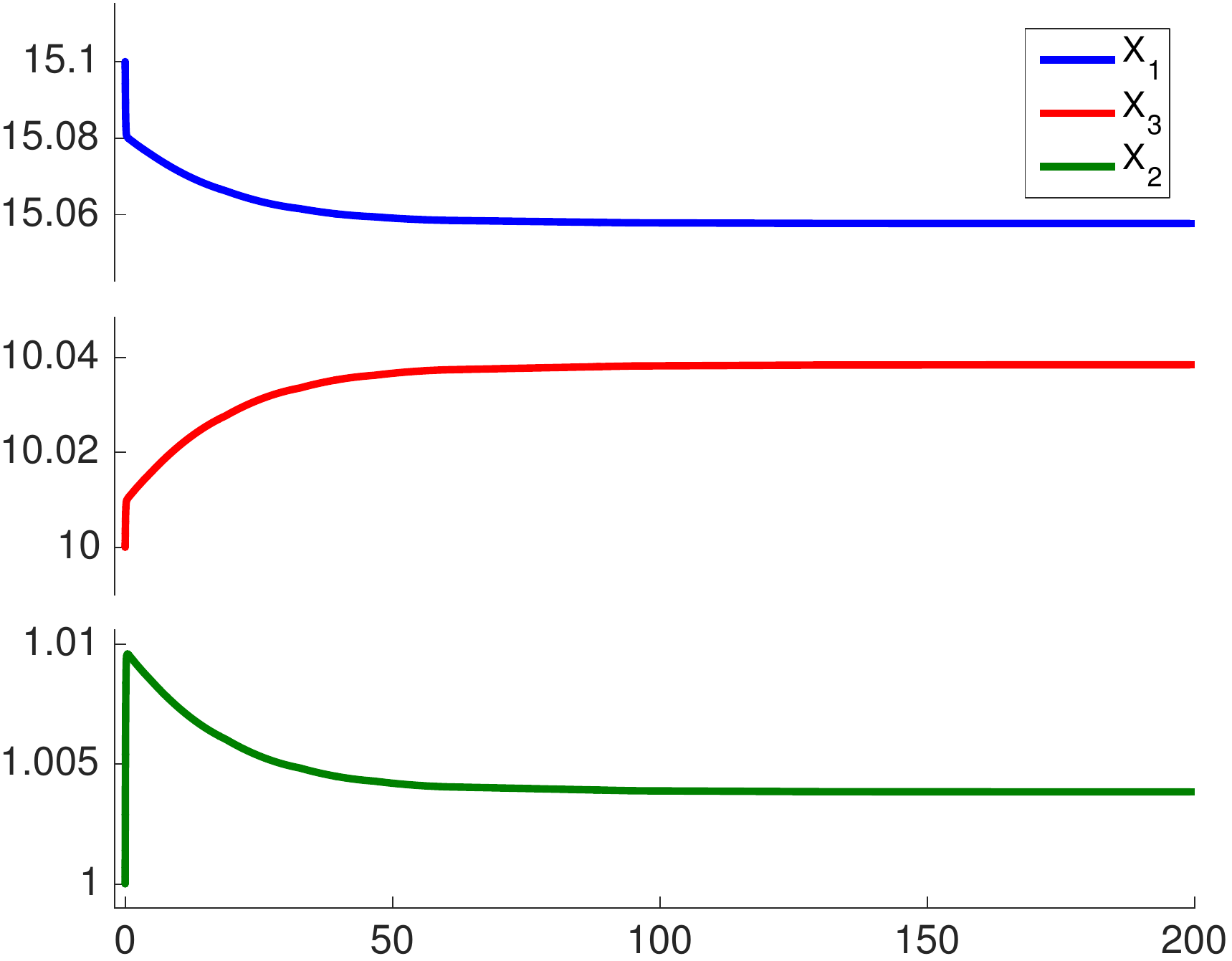}
\\[5\jot]
\caption{The solution of~\eqref{eq:ode-exampleA} in linear time (left) and in logarithmic time (right) using the same parameters as in Figure~\ref{fig:exampleA}, $\overline c = (15,1,10)$ and $k=(10,0.1)$. The initial condition is $c(0)=\overline{c} + \e e_1$ with $\e =0.1$.}
\label{fig:exA-c}
\end{figure}
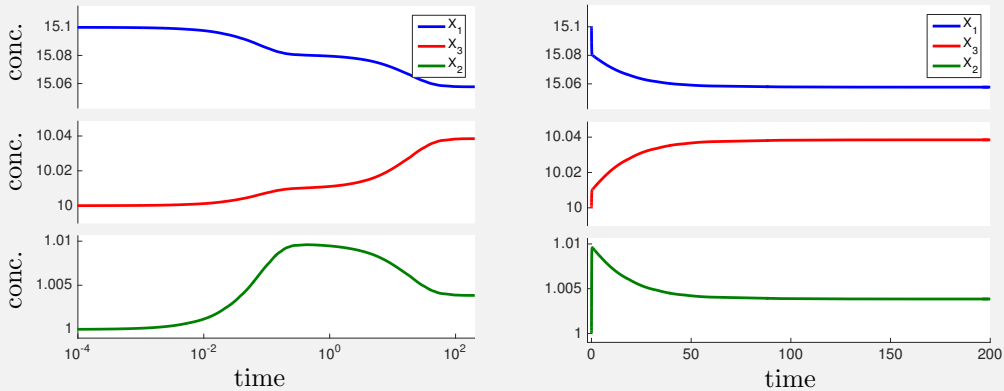
\begin{figure}[H]
  \begin{center}
  \labellist
  \small
  \pinlabel{time} [t] at 245 2
  \pinlabel{\rotatebox{90}{concentration}} [r] at -5 180
  \pinlabel {$S\approx1.44$} [b] at 273 345
  \pinlabel {$P^{-1}\approx0.58$} [l] at 468 168
  \endlabellist
  \includegraphics[width=.5\textwidth]{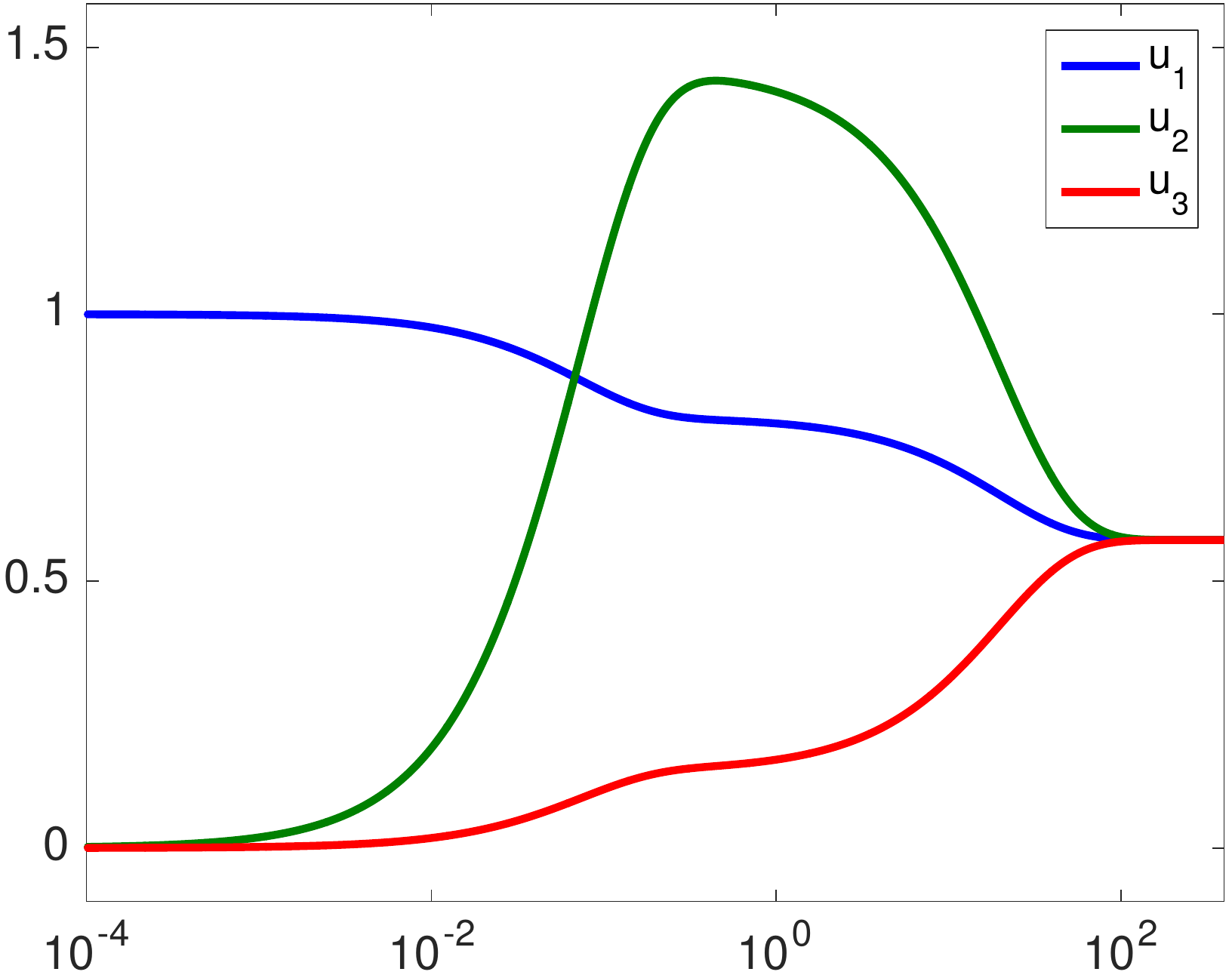} 
  \end{center}
  \vskip7\jot
  \caption{The solution of the equation $\dot u = Au$, $u(0)= (1,0,0)^T$, for the rescaled linearized concentrations $u_1,u_2,u_3$ of the species $X_1,X_2,X_3$. The parameters are the same as in Figure~\ref{fig:exampleA}, $\overline c = (15,1,10)$ and $k=(10,0.1)$.}
\label{fig:exA-u}
\end{figure}
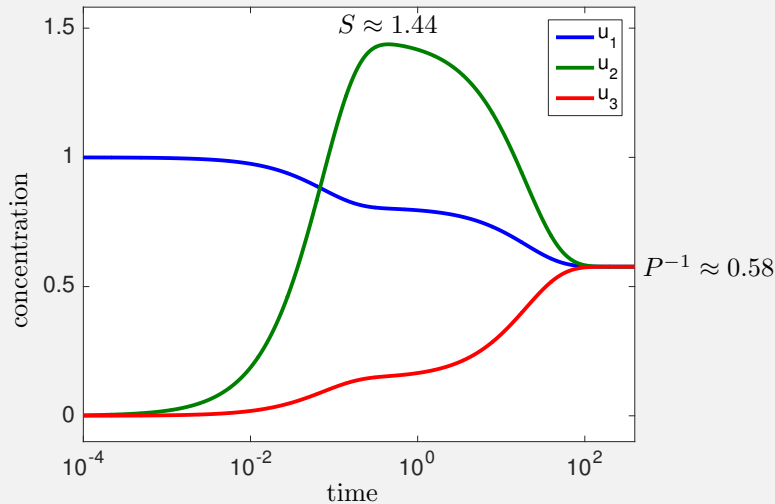
In the graph above one can read off the Sensitivity and Precision: choosing $X_2$ as output, the Sensitivity is the maximal value of  $u_2$ over all time (about $1.44$), and the Precision is the limiting value of $u_2$ as time tends to infinity ($0.58$). Note that the choice of input variable $i=1$ is encoded in the initial data for~$u$, and the choice of output variable $o=2$ means that we measure $u_2$.
\end{example}

\subsection{Maximization over parameters}

As we mentioned in the introduction, in much of this paper we take the position that the stoichiometry $(\sS,\sN,\sR)$ of a system is given, and we ask within which bounds we can make Sensitivity and Precision vary by the freedom of choosing coefficients $\overline c$ and $k$. This leads to the following three numbers:
\begin{align*}
\maxS(\sS,\sN,\sR) &:= \sup \{S: \ (\overline c,k)\in \R_+^\sS\times\R_+^\sR\Bigr\},\\
\maxP(\sS,\sN,\sR) &:= \sup \{P: \ (\overline c,k)\in \R_+^\sS\times\R_+^\sR\Bigr\},\\
\maxInvP(\sS,\sR,\sN) &:= \sup\Bigl\{P^{-1}:\ (\overline c,k)\in \R_+^\sS\times\R_+^\sR\Bigr\}.
\end{align*}
The maximal \emph{inverse} Precision plays a role in characterizing maximal Sensitivity (see Section~\ref{subsec:maxS-maxInvP}).

\begin{example}
\refstepcounter{example}
\label{ex:B}
\noindent
\textbf{Example \ref{ex:B}} (Arbitrarily large $S$ and $P$). We generalize Example~A by replacing $2X_1$ with $nX_1$ for some $n \in \N$ and a choice of $X_1$ for input and $X_3$ for output.
\begin{equation}
n X_1  \overset{k_1^+}{\underset{k_1^-}\rightleftharpoons } 
 X_2+X_3,   \qquad 
 X_2   \overset{k_2^+}{\underset{k_2^-}\rightleftharpoons }  X_3.  \nonumber
\end{equation} 
We want to explore the values of Precision and Sensitivity. Since the system is small we are able to explicitly calculate the Precision directly from Definition~\ref{def:precision} as follows. The stoichiometric matrix enforces that $2c_1 + n c_2  + n c_3 $ is constant along time trajectory of concentrations in every stoichiometric simplex. Therefore $2\overline{c}_1 + n \overline{c}_2  + n \overline{c}_3 = a$, where $a$  is a positive constant.  Let $\overline c^{\e}$ be the perturbed steady state when $\e$ is added to $X_1$ and then $2\overline c^{\e}_1 + n \overline c^{\e}_2  + n \overline c^{\e}_3 = a+2\e $. For both steady states we have
 $\sK(\overline{c})=\sK (c^{\e})= (0,0)^T$. With these relations the calculation of inverse Precision is straightforward:
\begin{equation}
\label{eq:PrecisionExampleB}
{P}^{-1}  = \frac{\overline{c}_1}{\overline{c}_3} \frac{d c^{\e}_3}{d \e}\bigg|_{\e = 0} = \dfrac{2n}{4+n^2\dfrac{\overline{c}_2}{\overline{c}_1}+n^2\dfrac{\overline{c}_3}{\overline{c}_1}}. 
\end{equation}
We observe that by choosing the ratio $({\overline c_2}+{\overline c_3})/{\overline c_1}$ large enough, one can have arbitrarily high Precision, i.e.\ $\maxP = \infty$.

This example is able to show Sensitivity arbitrarily close to $n$, i.e. $\maxS = n$. It is an easy exercise to show that the first reaction, as a subsystem, has $P^{-1}= n/(1+\overline{c}_3/\overline{c}_2+n^2\overline{c}_3/\overline{c}_1)$, and consequently $\maxInvP=n$. If we choose $(\overline c,k)$ in such a way that the first reaction happens much faster than the second, then $u_3$ can rise arbitrarily close to $n$ and after some time that the second reaction takes place, it comes down arbitrarily close to 0. Figure~\ref{fig:exB} shows the plot of $u_3$ for three values of $n$. 
\begin{figure}[H]
  \begin{center}
  \labellist
   \small
   \pinlabel{time} [t] at 207 2
   \pinlabel{\rotatebox{90}{concentration}} [r] at -2 180
  \endlabellist  \includegraphics[width=.4\textwidth]{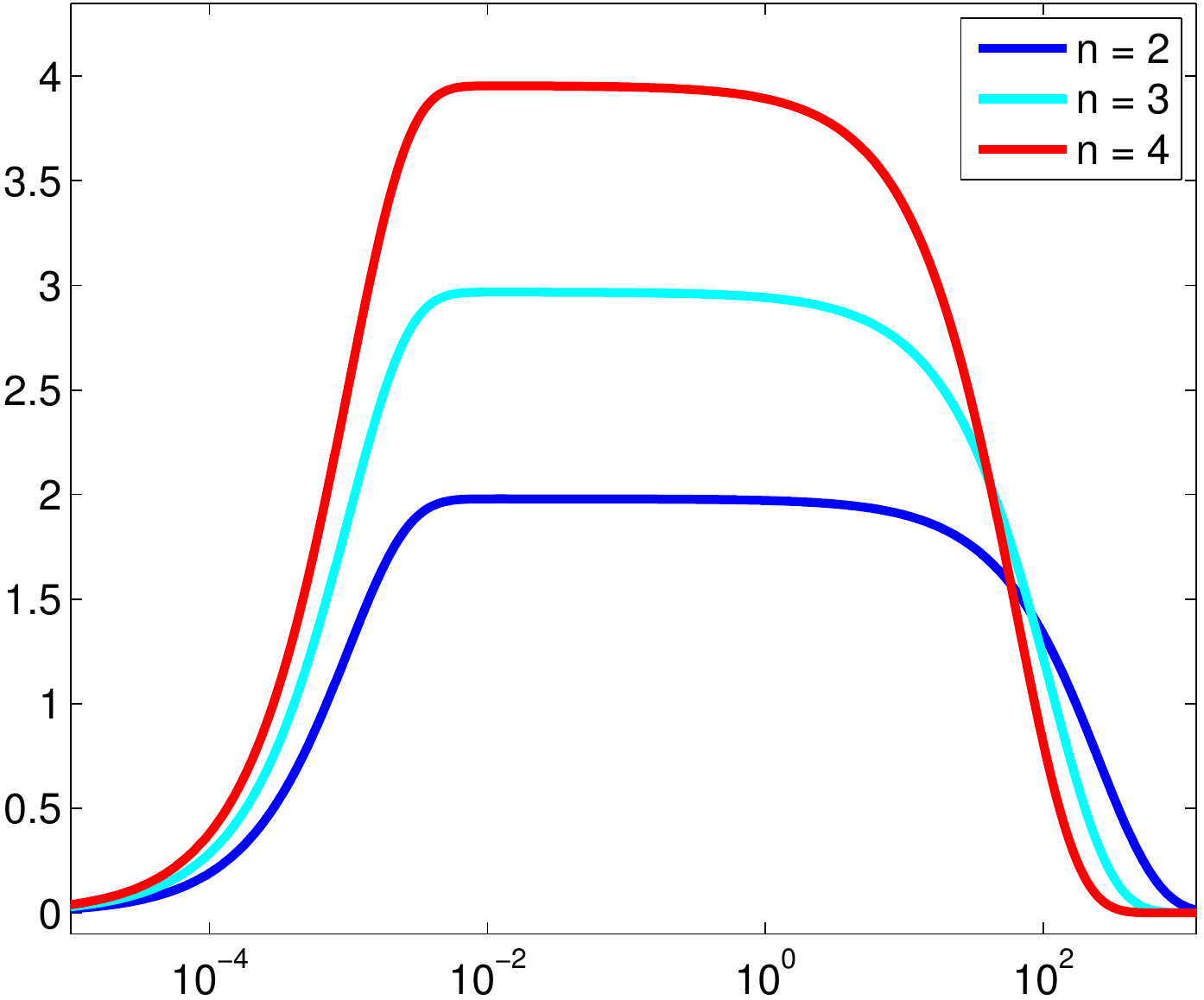} 
  \end{center}
  \caption{Plot of $u_3$ from the solution of the equation $\dot u = Au$, $u(0)= (1,0,0)^T$ for three values of $n$. The parameters are $\overline c = (10^2,10^5,10^{-2})$ and $k=(10,0.1)$.}
\label{fig:exB}
\end{figure}
\end{example}

\section{Results: Properties of the Precision}
\label{sec:precision}

As described above, the aim of this paper is to explore the degree to which detailed-balance, mass-action systems can have large Sensitivity and large Precision. In this section we focus on the Precision, and prove two main results. The first is an explicit formula for \emph{homogeneous} systems; the second is a characterization of the \emph{minimal} Precision in terms of the stoichiometry, which will be of use in Section~\ref{sec:sensitivity}.

We choose a system $(\sS,\sR,\sN,(\overline c,k))$, and we fix an input species $i\in \sS$ and an output species $o\in \sS$. 

\medskip

In some cases the Precision can be calculated explicitly. Example~\ref{ex:B} above is an instance of this; another instance is the class of \emph{homogeneous} systems. A reaction network is called \emph{homogeneous} of order $\kappa$ if for $\kappa \in \mathbb{N}$ fixed, all the reactions are of the type 
\[ 
\kappa  X_{s} \overset{k_j^+}{\underset{k_j^-}\rightleftharpoons }  \kappa X_{s'}. 
\]
Each column of $\sN$ in such reaction network consists of two nonzero elements with values  $\kappa$ and $-\kappa$. Therefore $(1,\hdots,1) \sN=0$, which implies that $\sum_s \overline{c}_s$ is constant in each stoichiometric simplex. Below we derive an explicit formula for the Precision of homogeneous systems.

\begin{theorem}[Precision for homogeneous systems]
\label{th:precision-homogeneous-systems}
If a reaction network is homogeneous of some order $\kappa$, then for any input and output
\begin{equation}
\label{eq:homogPrecision}
P = \frac{\sum_s \overline c_s}{\overline c_i}.
\end{equation}

\end{theorem}

\begin{example}
\refstepcounter{example}
\label{ex:uni}
\noindent
\textbf{Example~\ref{ex:uni}.}
Unimolecular reactions are a good example of a homogeneous reaction network. Below we present a reaction diagram between four species. Letting $X_1$ be the input, the inverse Precision for the three other species is $P^{-1} = 20/91\approx 0.22$ (see~\eqref{eq:homogPrecision}).
\begin{figure}[H]
\raisebox{-0.5\height}{%
 \labellist
  \small
  \pinlabel{time} [t] at 255 15
  \pinlabel{\rotatebox{90}{concentration}} [r] at -2 130
  \endlabellist
  \includegraphics[height=0.3\textwidth]{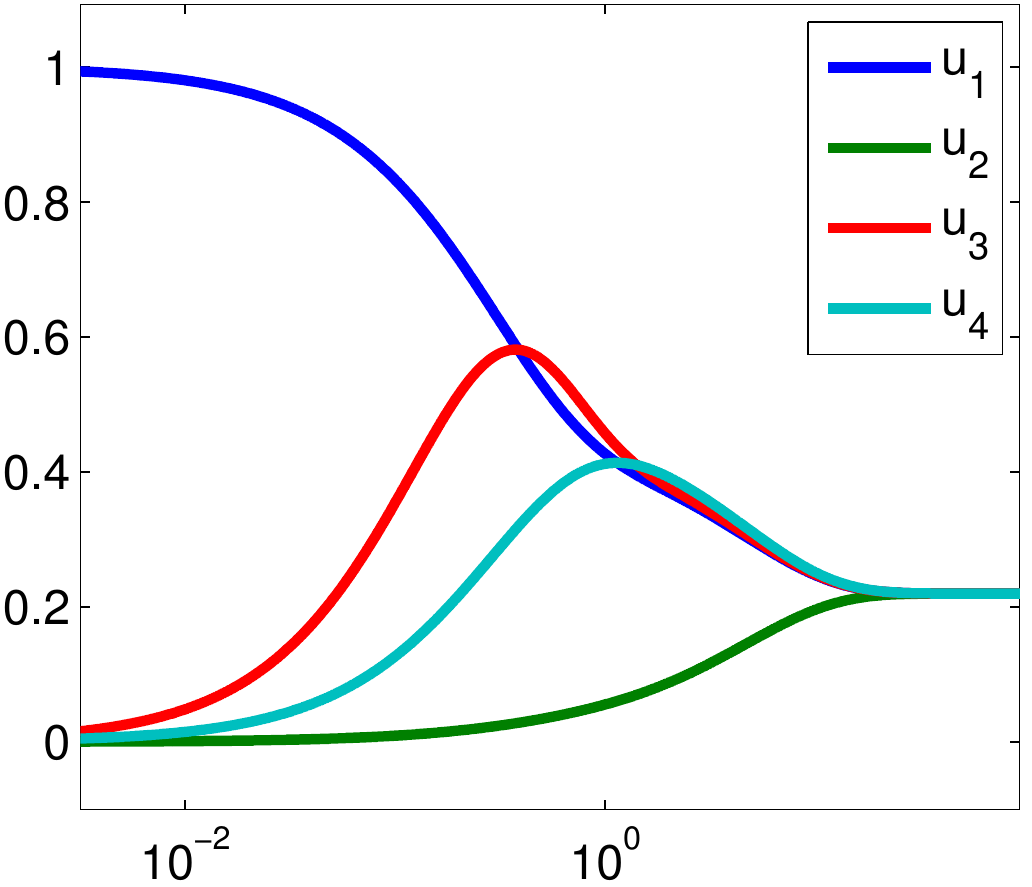}}
\qquad
\schemestart
    $X_1$\arrow(x1--x2){<=>[$k_1$]}[45] $X_2$ \arrow(--x3){<=>[$k_2$]}[-45] $X_3$   
    \arrow(@x1--x4){<=>[][$k_4$]}[-45] $X_4$ \arrow(--){<=>[][$k_5$]}[45] $X_3$   
      \arrow(@x1--@x3)  {<=>[$k_3$]}  
\schemestop
\vskip\jot
\caption{ Solution of $\dot{u}=Au$ (left) with $u(0)=(1,0,0,0)^T$, $\overline{c}=(20,50,1,20)^T$, and $k = (5,0.01,5,30,0.1)^T$. The reaction diagram (right).  }
\end{figure}
\end{example}

\begin{proof}
Fix the parameters $(\overline c,k)$. 
Equation~\eqref{eq:DBMA-kinetics} implies that at any steady state $\tilde c$, for a reaction involving species $s$ and $s'$, the ratio $\tilde{c}_{s}/\tilde c_{s'}$ equals $\overline c_s/\overline{c}_{s'}$ and thus is the same for each steady state. This implies that all steady states are multiples of each other. On the other hand we have a conservation law $\sum_s \overline{c}_s = a $ for some positive constant $a$. In view of the definition of Precision (following the notation of Section~\ref{subsec:Precison-Sensitivity-definition}) we have $\sum_s c_s^\e (\infty) = a+\e$. This can be written
as
 \[
 c_o^\e (\infty) = \dfrac{a+\e}{\sum_s \dfrac{c_s^\e (\infty)}{ c_o^\e (\infty)} },
 \]
in which the denominator is a sum of constant steady state ratios, hence independent of $\e$. Finally the inverse Precision is 
\begin{align*}
P^{-1}  =  \lim_{\e\to0} \frac{\log c^\e_o(\infty) - \log \overline c_o}{\log c^\e_i(0) - \log \overline c_i\strut}  
  =  \dfrac{\overline c_i}{\overline c_o} \dfrac{{d}c_o^\e (\infty) }{{d}\e}\Big|_{\e=0} 
  = \dfrac{\overline c_i}{\sum_s \overline{c}_s}. \nonumber
\end{align*}

\end{proof}

Note that the Precision for such a homogeneous network is always larger than one, and can be made arbitrarily large by tuning $\overline c$---specifically, by making the stationary concentration of the input variable small with respect to the other concentrations. This might seem like a good thing; however, we will see below that such a choice makes it difficult to have high \emph{Sensitivity}, and therefore `good' systems do not choose this route. 

\medskip

In fact, as we shall see in Section~\ref{sec:sensitivity}, there is a strong suggestion that having high Sensitivity requires \emph{low} Precision \emph{for a subsystem}. Because of this reason, it is interesting to consider \emph{lower} bounds on Precision, or equivalently, upper bounds on inverse Precision. In the rest of this section we characterize the \emph{maximal inverse Precision} for a given system $(\sS,\sR,\sN)$,
\[
\maxInvP(\sS,\sR,\sN) := \sup\Bigl\{P^{-1}:\ (\overline c,k)\in \R_+^\sS\times\R_+^\sR\Bigr\},
\]
in terms of the stoichiometry of the system, ie. in terms of $\sN$.

%

\medskip

The {\em support} of a vector $x\in\R^\sS$ is $\supp(x):=\{s\in \sS: x_s\neq 0\}$. Considering a linear space $W\subseteq \R^\sS$, 
we say that a vector $w\in W$ is  {\em elementary}  if $w$ is nonzero and  $\supp(w)$ is minimal in $W$, i.e. there exists no nonzero $w'\in W$ with $\supp w'\subsetneqq\supp w$. The {\em orthogonal complement} of $W$ is $$W^\perp:=\{u\in \R^\sS: u \perp w\text{ for all } w\in W\}.$$
In the following theorem, maximal inverse precision is characterized in terms of elementary vectors of $W$ and $W^\perp$, where  $W:=\range(\sN)$. Observe that if $u\in W^\perp$ and $c,c'\in G(\gamma)$, then $c'-c\in W$ and hence $u^Tc -u^Tc'=u^T(c-c)=0$; so each $u\in W^\perp$ describes an invariant linear combination $\sum_{s\in \sS} u_s c_s$ of the stoichiometric simplex $G(\gamma)$. 
\begin{theorem}[Sharp upper bounds on inverse Precision]
\label{th:UpperBoundsInvP}
Fix a system $(\sS,\sR,\sN)$ and input and output species $i,o\in \sS$, and let $W:=\range(\sN)$. 
Then:
\begin{subequations}
\label{equiv:maxInvP-formulations}
\begin{align}
\notag
\maxInvP &:= \sup\Bigl\{P^{-1}:\ (\overline c,k)\in \R_+^\sS\times\R_+^\sR\Bigr\}\\
&=\sup\Bigl\{ u_o :  u-e_i = \diag \Bigl(\frac1 {\overline{ c}}\Bigr) w,\ \overline { c}\in \R_+^\sS,\ w\in W,\ u\in W^\perp  \Bigr\}\\
&= \max \{ u_o:  u\text{ elementary in }W^\perp\text{ with } u_i=1, \text{ or }u=0\} 
  \label{equiv:maxInvP-formulations-Wperp}\\
&= \max \{ w_i :  w\text{ elementary in }W\text{ with } w_o=-1, \text{ or }w=0\} 
\label{equiv:maxInvP-formulations-W}
\end{align}
\end{subequations}
\end{theorem}

\begin{example}
\label{ex:ExA-cont-elementary-vectors}
\noindent
Recall \textbf{Example A}, where we already observed that $W = \range(\sN) = (1,1,1)^\perp$, and $W^\perp = \Span\{(1,1,1)\}$. Therefore all elementary vectors in $W^\perp$ are multiples of $(1,1,1)$, and the characterization~\eqref{equiv:maxInvP-formulations-Wperp} reduces to the maximum over two elements:
\[
\max\Bigl\{ u_2: u\in \{(1,1,1),(0,0,0)\}\Bigr\} = 1.
\]
For~\eqref{equiv:maxInvP-formulations-W}, the space $W$ has three elementary vectors, up to scalar multiples, which are $(1,-1,0)$, $(0,1,-1)$, and $(1,0,-1)$. Recall that the input variable is 1, and the output variable 2; therefore of these three directions, the third does not appear in the maximum, since it can not be rescaled to have $w_o=w_2 =-1$. The maxmimum in~\eqref{equiv:maxInvP-formulations-W} then reduces to
\[
\max \Bigl\{ w_1: w \in \{(1,-1,0),(0,-1,1),(0,0,0)\} \Bigr\} = 1.
\]
\end{example}

\begin{example}
\refstepcounter{example}
\label{ex:negP}
\noindent\textbf{Example \ref{ex:negP}.}
As an example where the alternative options $u=0$ and $w=0$ are relevant, consider  the single reaction
\[
X_1 + X_2 \leftrightharpoons X_3, \qquad i=1, o=2.
\]
Here $\sN = (1,1,-1)^T$, $W=\Span\{(1,1,-1)\}$, and $W^\perp = (1,1,-1)^\perp$; therefore $W$ has only the elementary vector $(1,1,-1)$, up to scalar multiplication, and~\eqref{equiv:maxInvP-formulations-W} reduces to 
\[
\max\Bigl\{ w_1: w \in \{(-1,-1,1),(0,0,0)\} \Bigr\} = 0.
\]
In this case $W^\perp$ has elementary vectors $(1,-1,0)$, $(1,0,1)$, and $(0,1,-1)$, so that~\eqref{equiv:maxInvP-formulations-Wperp} becomes
\[
\max\Bigl\{ u_2: u\in \{(1,-1,0),(1,0,1), (0,0,0)\}\Bigr\} = 0.
\]
Indeed, the reaction has negative Precision for all positive values of $\overline c$ (since increase in $X_1$ always leads to decrease in $X_2$.) Therefore the max inverse Precision can reach zero, but can not be positive.
\end{example}
\bigskip
The maxima \eqref{equiv:maxInvP-formulations-Wperp} and \eqref{equiv:maxInvP-formulations-W} may be evaluated by enumerating the elementary vectors of a linear space. We will describe an algorithm for this in Section \ref{sec:matroid}.

The rest of this section is devoted to the proof of this theorem, through a series of lemmas. The first lemma provides the connection between inverse Precision on one hand and the vectors~$u$ and~$w$ that appear in Theorem~\ref{th:UpperBoundsInvP}.

\begin{lemma}\label{lem:P_o}
For given $\overline c\in \R_+^\sS$, the conditions 
\begin{equation}
\label{eq:u}
u\in W^\perp, \quad w\in W, \quad u = e_i+\diag\Bigl(\frac1{\overline{c}}\Bigr)w, 
\end{equation}
uniquely determine the pair $(u,w)$. Then
\begin{equation}
\label{char:P-u}
P^{-1} = u_o.
\end{equation}
As a consequence,
\[
\maxInvP =
\sup\left\{u_o:  u-e_i=\diag\Bigl(\frac1{\overline{c}}\Bigr)w, ~u\in W^\perp, ~w\in W, ~\overline{c}\in \R_+^\sS\right\}.
\]
\end{lemma}

\begin{proof}
The existence of a $u$ satisfying~\eqref{eq:u} will follow from the argument in the next paragraph; here we show that $u$ satisfying~\eqref{eq:u} is unique. If not, then there are distinct $u, u'$ satisfying \eqref{eq:u}. Then $v:= u-u'$ is a nonzero vector such that $v\in W^\perp$ and  $v\in \diag(\frac1{\overline{c}})W$, so that $\diag( \overline c)v \in W$, and consequently $v\perp \diag (\overline c)v$. Since $\overline c$ has strictly positive components, this implies $v=0$, a contradiction. 

We now show~\eqref{char:P-u}. For each $\e>0$, write $c_\e := \hat c[\overline c+ \e e_i]$ for the unique stationary state in the same stoichiometric simplex as $\overline c+\e e_i$. By Lemma~\ref{lem:EL}, $c_\e$ is a smooth function of $\e$; we write $\dot c_\e := dc_\e/d\e \big|_{\e=0}$, which is the vectorial rate of change of the stationary state as we add component~$X_i$. Again by Lemma \ref{lem:EL}, $\sN^T \log (c_\e/\overline c)=0$, and by differentiating we find $\sN^T (\dot c_\e/\overline c)=0$. From $\hat c[\gamma]\in \gamma+W$ follows  $\hat c[\overline c+\e e_i] - \overline c\in \e e_i + W$, and therefore we find that $\dot c_\e$ satisfies the two equations
\[
\sN^T \frac {\dot c_\e}{\overline c} = 0, \qquad \dot c_\e \in e_i + W.
\]
Defining $u = \overline c_i (\dot c_\e/\overline c) = \overline c_i \diag (1/\overline c) \dot c_\e$ these can be rewritten as
\[
\sN^T u = 0, \qquad u \in \overline c_i \diag\Bigl(\frac1{\overline c}\Bigr) (e_i+ W) = e_i + \diag\Bigl(\frac1{\overline c}\Bigr)W,
\]
which is equivalent to~\eqref{eq:u}. This also proves the existence of a solution to~\eqref{eq:u}. The fact that $P^{-1}= u_o$ (equation~\eqref{char:P-u}) is then a direct consequence of \eqref{altdef:P}.
\end{proof}

\begin{remark}[Chemical interpretation of $u$ and $w$]
The proof of this lemma illustrates the chemical interpretation of $u$ and $w$. Both are defined in terms of a curve of stationary states $c_\e = \hat c[\overline c+\e e_i]$ generated by perturbing the system by adding small amounts of $X_i$:
\begin{enumerate}
\item $u$ can be interpreted as the {\boldmath derivative of the vector function $\e\mapsto \overline c_i\log c_\e $ at $\e=0$}, i.e. $\overline c_i$ times the \textbf{rate of change of the vector of chemical potentials of the species};
\item $w$ can be interpreted as the \textbf{(infinitesimal) stoichiometrically admissible perturbation} that connects the non-stationary point $\overline c+\e e_i$ with the stationary point $c_\e = \hat c[\overline c+\e e_i]$: $w =\overline c_i \lim_{\e\to0} \e^{-1}( \hat c[\overline c+\e e_i]-\overline c - \e e_i)$.
\end{enumerate}
\noindent
\end{remark}

In the following lemmas we characterize the \emph{maximal} inverse precision, where the maximum is taken over all $\overline c\in \R^\sS_+$, in terms of elementary vectors in $W$ and $W^\perp$. Lemmas~\ref{lem:o_u}, \ref{lem:u_o}, and~\ref{lem:u_w} together conclude the proof of Theorem~\ref{th:UpperBoundsInvP}.

\begin{lemma}\label{lem:o_u}
\begin{align*}
\sup&\left\{u_o:  u-e_i =  \diag\Bigl(\frac1{\overline{c}}\Bigr)w, ~u\in W^\perp, ~w\in W, ~\overline{c}\in \R_+^\sS\right\}\\
&\leq \max \{ u_o:  u\text{ elementary in }W^\perp\text{ with } u_i=1, \text{ or }u=0\} 
\end{align*}
\end{lemma}
\begin{proof} 
The condition 
\begin{equation}\label{eq:orth}u-e_i=\diag\Bigl(\frac1{\overline{c}}\Bigr)w\quad \text{for some }~w\in \R_+^\sS\end{equation}
implies for  $s\neq i$ that if $w_s>0$, then $u_s>0$; if $w_s<0$, then $u_s<0$; and if $w_s=0$, then $u_s=0$.
Moreover, if $u_i(u_i-1)>0$, then 
$$0<u_i\overline{c}_i(u_i-1)+\sum_{s\neq i} u_s\overline{c}_su_s=u^T\diag(\overline{c})(u-e_i)=u^Tw=0,$$
a contradiction. Hence $0\leq u_i\leq 1$.
Summarizing, we find that \eqref{eq:orth} implies
\begin{equation}
\label{eq:lin}
u_s\left\{
  \begin{array}{ll} 
    \in [0,1] &\text{if }s=i\\
    \geq 0&\text{if } w_s>0, s\neq i\\
    \leq 0&\text{if } w_s<0, s\neq i\\
    = 0&\text{if } w_s=0, s\neq i 
  \end{array}
\right.
\end{equation}
So for fixed $w\in W$, if $w_o=0$, then $u_o=0$ and the result is trivial. Otherwise the supremum
\begin{equation}\label{eq:fix_w}\sup\left\{u_o:  u-e_i=\diag\Bigl(\frac1{\overline{c}}\Bigr)w, ~u\in W^\perp, ~\overline{c}\in \R_+^\sS\right\}\end{equation}
is bounded from above by 
\begin{equation}\label{eq:fix_w2}\sup\{u_o: u\in W^\perp, u\text{ satisfies } \eqref{eq:lin}\}.\end{equation}

This is a linear optimization problem. To complete the proof of this lemma, we will argue that for each fixed $w\in W$, \eqref{eq:fix_w2}, and hence \eqref{eq:fix_w}, is bounded from above by $$\max \{ u_o:  u\text{ elementary in }W^\perp\text{ with }u_i=1,\text{ or }u=0\}.$$
If the supremum in~\eqref{eq:fix_w2} equals $+\infty$, then there exist $u$ and $v$ such that $u+\lambda v\in W^\perp$ satisfies \eqref{eq:lin} for all $\lambda>0$, with $v_o>0$ and $v_i=0$. But then $w^Tv=\sum_s w_sv_s\geq w_ov_o>0$, contradicting that $v\perp w$. So the value of \eqref{eq:fix_w2} is finite, and hence the optimum is attained. Let $u^*$ be an optimal solution such that $|\supp(u^*)|$ is as small as possible. 

We show that $u^*=0$ or $u^*$ is an elementary vector. If not, then $u^*\neq 0$ and there is a nonzero vector $v\in W^\perp$ so that $\supp(v)\subsetneqq \supp(u^*)$, by the definition of an elementary vector. We may assume that $v_o=0$; otherwise replace $v$ with $ u^*_o v-v_o u^*$. Then 
$$\max\{u_o: u=u^*+\epsilon v, \epsilon\in \R, u\text{ satisfies } \eqref{eq:lin}\}$$
is attained by an optimal solution $u^{**}$ which satisfies one of the inequalities in \eqref{eq:lin} with equality, such that $u^{**}_s=0$ where $u^*_s\neq 0$ for some $s$. Since $\supp(u^{**})\subseteq \supp(u^*)\cup\supp(v) \subseteq \supp(u^*)$ and $u^{**}_o=u^*_o+\epsilon v_o= u^*_o$, that would contradict the choice of $u^*$ as an optimal solution of \eqref{eq:fix_w2} with minimal support.

So $u^*$ is elementary, and it remains to show that $u^*_i=1$. If $u^*_i=0$, we have  $0=w^Tu^*=\sum_{s\neq i} w_su^*_s$, and hence $u^*=0$. Hence $0<u^*_i\leq 1$. If $u^*_i<1$, then for a sufficiently small $\epsilon>0$, the vector $(1+\epsilon)u^*$ is a feasible solution of \eqref{eq:fix_w2} with $((1+\epsilon)u^*)_o> u^*_o$, contradicting the optimality of $u^*$. Hence $u^*_i=1$, as required.
\end{proof}

We need the following, essentially combinatorial fact on elementary vectors. For any $ \sS'\subseteq \sS$ and $\sR' \subseteq \sR$, let $\sN[\sS', \sR']$ denote the submatrix of  $\mathcal N$ spanned by the rows and columns in $\sS'$ resp. $\sR'$.
\begin{lemma}\label{lem:basic}
Let $u^*$ be elementary in $W^\perp$. Then there exist a set $\sS'\subseteq \sS$, elementary vectors $\sm w^s\in W$ for $s\in \sS'$, and elementary vectors $\sm u^s\in W^\perp$ for $s\in \sS\setminus \sS'$, such that 
\begin{itemize}
\item $i\not\in \sS'$, and $\supp(u^*)\subseteq \sS'\cup\{i\}$;
\item $\supp(\sm w^s)\subseteq (\sS\setminus \sS')\cup \{s\}$ for each $s\in \sS'$, and $\sm w_s^s=1$; and 
\item $\supp(\sm u^s)\subseteq \sS'\cup \{s\}$ for each $s\in \sS\setminus \sS'$, and $\sm u_s^s=1$.
\end{itemize}
\end{lemma}
\proof Consider the set $\mathcal U:=\supp(u^*)\setminus \{i\}$. The rows of $\sN[\mathcal U, \sR]$ are independent, for if there were  a linear dependency among these rows, then there would exist a nonzero vector $u'\in  \ker(\sN^T)=W^\perp$ with $\supp(u')\subseteq \mathcal U$, contradicting our assumption that $u^*$ is elementary vector of $W^\perp$. 

Pick any maximal set $\sS'\subseteq \sS$ so that $\mathcal U\subseteq  \sS'$ and so that the rows of $\sN[\sS', \sR]$ are still independent. Then the rows of $\sN[\sS', \sR]$ form a basis of the rowspace of $\sN$. Let $\sR'\subseteq \sR$ be such that the columns of $\sN[\sS, \sR']$ are a basis of $W$. Then by applying column operations to $\sN[\sS, \sR']$ (as in standard Gaussian elimination) we may obtain a matrix $\sN' \sim \sN[\sS, \sR']$ of the form
$$\bordermatrix{~ & \sR' \cr \sS' & I \cr \sS'' & X}.$$
That is, $\sN'[\sS',  \sR']=I$, and as column operations do not change the column space, $\range(\sN')=\range(\sN)=W$. For each $s\in \sS$, let $\sm w_s$ denote the unique column of $\sN'$ with a 1 in the $s$-th row. Then $\sm w_s^s=1$ and $\supp(\sm w^s)\subseteq (\sS\setminus \sS')\cup \{s\}$ by construction, and moreover $\sm w^s$ is elementary: any $w\in W$ is a linear combination of the columns of $\sN'$, so that if $\supp(w)\subseteq \supp(\sm w^s)$, then $w$ must be a scalar multiple of $\sm w^s$.

To obtain the vectors $\sm u^s$, we construct the matrix $\sN''$ as
$$\bordermatrix{~ & ~ \cr \sS' & -X^T \cr \sS'' & I },$$
using the matrix $X:=\sN'[\sS'', \sR']$. Since $(\sN')^T\sN''=0$ and $\rank(\sN')+\rank(\sN'')=|\sS'|+|\sS''|=|\sS|$, the columns of $\sN''$ span $W^\perp$. 
For   each $s\in \sS\setminus \sS'$, let $\sm u^s$ denote the unique column of $\sN''$ with a 1 in the $s$-th row. Then  $\sm u_s^s=1$, 
$\supp(\sm u^s)\subseteq \sS'\cup \{s\}$, and $\sm u^s$ is elementary as before.
\endproof
We comment on the relation with matroid theory in Section \ref{sec:matroid}.

\begin{lemma}
\label{lem:u_o}
\begin{align*}
\sup\left\{u_o:  u-e_i =  \diag\Bigl(\frac1{\overline{c}}\Bigr)w, ~u\in W^\perp, ~w\in W, ~\overline{c}\in \R_+^\sS\right\}\\
 \geq \max \{ u_o:  u\text{ elementary in }W^\perp\text{ with }u_i=1,\text{ or }u=0\}
\end{align*}
\end{lemma}
\begin{proof}
First, note that the supremum is necessarily nonnegative, since if $(u, w)$ is feasible in combination with some $\overline{c}\in \R_+^\sS$, then so is $(\alpha u, w)$ for any $0<\alpha <1$. So it remains to show that if  $u^*\in W^\perp$ is an elementary vector with $u^*_i=1$, then the supremum is at least $u^*_o$. As we have already established that the supremum is nonnegative, we may assume $u^*_o>0$.

To prove that the supremum is at least $u^*_o$, it suffices to show that for any $\epsilon>0$, there are $u\in W^\perp, w\in W$ such that $\|u-u^*\|<\epsilon$ and
$$u-e_i =  \diag\Bigl(\frac1{\overline{c}}\Bigr)w, ~u\in W^\perp, ~w\in W, \text{ for some }\overline{c}\in \R_+^\sS.$$
This condition is equivalent to 
\begin{equation}
\label{eq:signs}
\sign(u_s)=\sign(w_s)\text{ for all } s\neq i,\ 
\sign(u_i-1)=\sign(w_i), ~u\in W^\perp, ~w\in W.\end{equation}
We will construct such vectors $u, w$, using a set $\sS'$ and vectors $\sm u^s$ and $\sm w^s$ as in Lemma \ref{lem:basic}. 
Throughout, we will preserve that
\begin{align}\label{eq:weak_signs} \sign(u_s)=\sign(w_s) &\text{ for all } s\in \sS'\cap\supp(w)\\
\label{eq:weak_signs2} \sign(u_s)=\sign(w_s) &\text{ for all } s\in (\sS\setminus \sS'\setminus\{i\})\cap\supp(u)\\
\label{eq:weak_signs3} \sign(u_i-1)=\sign(w_i),& ~u\in W^\perp, ~w\in W.\end{align}
In each step, we increase the cardinality of $|\supp(u)\cap \supp(w)|$, until we attain $\supp(u)=\supp(w)$. Then, we necessarily have \eqref{eq:signs}.

We initialise by setting 
\begin{equation}
\label{alg:lem-u_o-initialization}
u\leftarrow (1-\delta_i)u^*\text{ for a small }\delta_i>0, \text{ and }w\leftarrow u^*_o\sm w^o.
\end{equation}
To see that \eqref{eq:weak_signs} holds for this initial $u,w$, note that $\sS'\cap \supp(w)=\sS'\cap \supp(\sm w^o)=\{o\}$,  so that we need only verify that $\sign(u_o)=\sign(w_o)$. Since $\sm w^o_o=1$, we have $\sign(w_o)=\sign(u^*_o\sm w^o_o)= \sign(u^*_o)=\sign(u_o)$, as required.
As $\supp(u)=\supp(u^*)\subseteq \sS'\cup\{i\}$, condition \eqref{eq:weak_signs2} is vacuously satisfied by $u,w$.
It remains to show \eqref{eq:weak_signs3}, that $\sign(u_i-1)=\sign(w_i)$. We have $\supp(u^*)\cap \supp(\sm w^o)=\{i,o\}$, $u^*\perp \sm w^o$, and hence $u^*_o\sm w^o_o+u^*_i\sm w^o_i=0$; moreover $u^*_i=1$, $u^*_o>0$, and $\sm w^o_o=1$, so that $\sm w^o_i<0$. So $w_i=u^*_o\sm w^o_i<0$, and as $u^*_i=1$, we have $u_i-1=-\delta<0$. Hence $\sign(u_i-1)=\sign(w_i)$, as required.

In the general step, if there is an $s\in \supp(w)\setminus\supp(u)$, we put
$$u\leftarrow u+\delta_s \sm u^s$$
where $\delta_s\in \R$ is chosen such that $\sign(\delta_s)=\sign(w_s)$ and with $|\delta_s|$ sufficiently small to ensure that for each $s'\not = s$ with $u_{s'}\neq 0$, the sign of $u_{s'}$ is unaltered. Then after this step, we have $s\in \supp(w)\cap\supp(u)$, and  \eqref{eq:weak_signs},  \eqref{eq:weak_signs2}, and \eqref{eq:weak_signs3} are preserved. If there is an $s\in \supp(u)\setminus\supp(w)$, we similarly put
\[
w\leftarrow w+\delta_s \sm w^s
\]
with $\sign(\delta_s)=\sign(u_s)$ and $|\delta_s|$ sufficiently small.

Since  each $\delta_s$ may be chosen arbitrarily close to $0$, we can ensure that in the final stage
\[
\|u-u^*\|= \biggl\|-\delta_i u^*  + \sum_{s\in \sS\setminus\sS'\setminus\{i\}} \delta_s \sm u^s \biggr\|<\epsilon.
\]
\end{proof}

\begin{lemma}\label{lem:u_w}
\begin{align*}
 \max \{ u_o:  u\text{ elementary in }& W^\perp,\  u_i=1\} =\max \{ w_i :  w\text{ elementary in }W,\ w_o=-1\} 
\end{align*}
\end{lemma}
\begin{proof} We first prove `$\leq$'. Let $u$ attain the maximum on the left. By Lemma \ref{lem:basic}, there is a set $\sS'\subseteq \sS$ and vectors $\sm w^s\in W$ so that $i\not\in \sS$, and $\supp(u)\subseteq \sS'\cup\{i\}$, and $\supp(\sm w^s)\subseteq (\sS\setminus \sS')\cup \{s\}$ for each $s\in \sS'$, and $\sm w_s^s=1$. Pick $w:=-\sm w^o$. Then $w_o=-1$, and $\supp(u)\cap\supp(w)=\{i,o\}$. As $w\perp u$ , we have
$$w_i-u_o=w_iu_i+w_ou_o=\sum_s w_su_s=0$$
As $w$ is elementary, $w$ is a feasible solution of the maximum on the right. Hence `$\leq $'.
Interchanging $W$ with $W^\perp$, and $i$ with $o$, we obtain the converse inequality `$\geq $'.
\end{proof}

This completes the proof of Theorem \ref{th:UpperBoundsInvP}.

\begin{example}
\refstepcounter{example}
\label{ex:Rudi}
\noindent
\textbf{Example \ref{ex:Rudi}.}  We illustrate the proof by considering a system with 6 species $\sS=\{X_1,\ldots, X_6\}$ and reactions
$$X_2 + X_3 \leftrightharpoons X_4, \qquad X_2 + 2X_5 \leftrightharpoons 2X_3 + X_6,\qquad 2X_1+X_2\leftrightharpoons X_6.$$
\ignore{Then the stoichiometric matrix is 
$$\sN:= \left(\begin{array}{rrr}
0 & 0 & 2 \\
1 & 1 & 1 \\
1 & -2 & 0 \\
-1 & 0 & 0 \\
0 & 2 & 0 \\
0 & -1 & -1
\end{array}\right) 
$$}
The stoichiometric space $W$ and its orthogonal complement $W^\perp$ contain the following elementary vectors:
$$
\begin{array}{r|rrrrrr} 
W& X_1&X_2&X_3&X_4&X_5&X_6\\ 
\hline
w^1&0&1&1&-1&0&0\\
w^2&0&1&-2&0&2&-1\\
w^3&2&1&0&0&0&-1\\
w^4&3&0&0&1&-1&-1\\
w^5&0&3&0&-2&2&-1\\
w^6&0&0&3&-1&-2&1\\
w^7&1&0&1&0&-1&0\\
w^8&2&0&-1&1&0&-1\\
w^9&1&-1&0&1&-1&0
\end{array}\qquad
\begin{array}{r|rrrrrr} 
W^\perp &X_1&X_2&X_3&X_4&X_5&X_6\\ 
\hline
u^1&1&-2&2&0&3&0\\
u^2&1&0&0&0&1&2\\
u^3&1&-2&0&-2&1&0\\
u^4&1&-2&-1&-3&0&0\\
u^5&1&1&-1&0&0&3\\
u^6&0&1&0&1&0&1\\
u^7&0&1&-1&0&-1&1\\
u^8&1&0&-1&-1&0&2\\
u^9&0&0&1&1&1&0\\
\end{array}
$$
These lists are complete, that is, each elementary vector of $W$ (resp. $W^\perp$) is obtained by scaling one of the vectors $w^k$ (resp. $u^k$). With input $i=X_1$ and output $o=X_5$, inspection of both tables reveals that the maxima
\begin{equation}
\label{exE:maxismax}
\max \{ u_o:  u\text{ elementary in } W^\perp,\  u_i=1\}=3=\max \{ w_i :  w\text{ elementary in }W,\ w_o=-1\}
\end{equation}
are attained by the elementary vectors $u^1\in W^\perp$ and $w^4\in W$, respectively. 

Using the algorithm of Lemma~\ref{lem:u_o}, we construct vectors $u\in W^\perp$, $w\in W$, $\overline{c}$ which are feasible in
$$\sup\left\{u_o:  u-e_i =  \diag\Bigl(\frac1{\overline{c}}\Bigr)w, ~u\in W^\perp, ~w\in W, ~\overline{c}\in \R_+^\sS\right\}$$
and such that $u_o$ is arbitrarily close to the maximum $u^1_o=3$. 

The vector $u^1$ has $\supp(u^1)=\{X_1, X_2, X_3, X_5\}$, and we use $\sS'=\{X_1, X_2, X_3\}$ in the algorithm. 
 In the initial step, we put $u\leftarrow (1-\delta_5)u^1, w\leftarrow -3w^4$, where $0<\delta_5\ll 1$. Then
$\supp(u)=\{X_1, X_2, X_3, X_5\}$, $\supp(w)=\{X_1, X_4, X_5, X_6\}$, and we have
 $(u_{X_1}-1)w_{X_1}>0$, and $u_{X_5}w_{X_5}>0$. 
To repair that $u_{X_2}<0$ and $w_{X_2}=0$, we use the elementary vector $w^5\in W$ with $\supp(w^5)\subseteq \{X_2\}\cup \sS'$. We have $w^5_{X_2}=3>0$, so we put $w\leftarrow w-\delta_2 w^5$ where $0<\delta_2\ll\delta_5$.  Then, we consider that $u_{X_4}=0$ and $w_{X_4}<0$, and add a small multiple of $u^4$ to compensate: $u\leftarrow u+\delta_4u^4$ with $0<\delta_4\ll\delta_2$. Next, we have $u_{X_3}>0$ and $w_{X_3}=0$, and so we add a small multiple of $w^6: w\leftarrow w+\delta_3w^6$ where $0<\delta_3\ll\delta_4$. Finally, we repair that $u_{X_6}=0$ whereas $w_{X_6}>0$, and put $u\leftarrow u+\delta_6u^5$ with $0<\delta_6\ll\delta_3$. 

We end up with 
$$u=(1-\delta_5)u^1+\delta_4u^4+\delta_6u^5\text{ and } w=-3w^4-\delta_2 w^5+\delta_3w^6$$
where $0<\delta_6\ll\delta_3\ll\delta_4\ll\delta_2\ll\delta_5\ll 1$. Constructing $\overline{c}$ by setting $\overline{c}_s=w_s/u_s$ for $s\neq i$ and $\overline{c}_i=w_i/(u_i-1)$, we obtain the feasible triple $w,u,\overline{c}$ as follows:
$$\begin{array}{rcccccc} 
 &X_1&X_2&X_3&X_4&X_5&X_6\\ 
\hline
w= & -9 & -3\delta_2 & 3\delta_3 &-3+2\delta_2-\delta_3&3-2\delta_2-2\delta_3 &3+\delta_2+\delta_3\\
u= & \delta'_5+\delta_4+\delta_6 & -2\delta'_5-2\delta_4+\delta_6 &2\delta'_5-\delta_4 -\delta_6&-3\delta_4 &3\delta'_5 &3\delta_6\\
\overline{c}\approx& 9\delta_5^{-1} & \frac{3}{2} \delta_2 & \frac{3}{2}\delta_3&\delta_4^{-1} & 1&\delta_6^{-1}\\
\end{array}
$$
Here we abbreviated $\delta'_5:=(1-\delta_5)$, and the approximation of $\overline{c}$ is based on the assumed relative magnitudes of the $\delta_s$. The objective value of $u_0=3-3\delta_5$ tends to the maximum 3 as $\delta_5\downarrow 0$, and at the same time $\overline{c}\rightarrow (\infty,0,0,\infty, 1, \infty)$.
 
One way the construction of $\overline c$ above can be interpreted is as follows. The optimum in~\eqref{exE:maxismax} is achieved in $u^1$ and $w^4$. Focusing on the $w$-side of this characterization, first note that each of the elementary vectors $w^i$ is uniquely characterized by its support,  up to multiplication by scalars, by the very definition of an elementary vector. Therefore the optimal vector $w^4$ is characterized by its zeros for coordinates $X_2$ and $X_3$. One can now force the system to  follow~$w^4$ by choosing~$\overline c$ such that $\overline c_2$ and $\overline c_3$ are much smaller than the other coordinates. Although in a \emph{relative} sense $X_2$ and $X_3$ participate in the reactions---as illustrated by the values in the table above, which give $u_2\approx -2$ and $u_3\approx2$---because of the low background concentrations they play no role in terms of \emph{absolute} concentrations (as illustrated by the low values of $w$). The end result is that the system becomes similar to a single equation with stoichiometry $w^4$, with small perturbations of other reactions. The construction above makes this statement concrete. 
\end{example}

\ignore{
(It would be nice if there were a simple, direct argument for the following Lemma, but right now I do not see one:
\begin{lemma} Let $w\in W$ be an elementary vector such that $i, o\in\supp(w)$. Then for any $\epsilon>0$, we have $P^{-1}+\epsilon\geq -w_i/w_o$, for some $\overline{c}\in \R^\sS_+$ and  $\gamma\in\R^\sR_+$.\end{lemma}
This Lemma has an obvious interpretation (`any elementary reaction can be induced by some choice of $\overline{c}, \gamma$') which motivates the concept of  an `elementary $w\in W$'. Stating this Lemma before Theorem \ref{th:UpperBoundsInvP} would make a better introduction to the Theorem.
)

\begin{remark}
 If $w$ is elementary in $W$, then necessarily $\dim\{w'\in W: \supp(w')=\supp(w)\}=1$. Hence for each subset $\sS'\subseteq \sS$, there can be at most one vector $w$ such that $w$ is elementary, $w_o=-1$, and $\supp(w)=\sS'$. The set $\{ w_i :  w\text{ elementary in }W, w_o=-1\}$ is thus finite, and computing $\max\{ w_i :  w\text{ elementary in }W, w_o=-1\}$ with reasonable efficiency becomes a combinatorial problem. We will comment on methods for computing this bound and the connections of this issue with matroid theory in section \ref{sec:matroid}.
\end{remark} 
}

\section{Results: Properties of the Sensitivity}
\label{sec:sensitivity}

We now turn to the Sensitivity. In contrast to the Precision, the Sensitivity is a dynamic property, that depends not only on the stationary state $\overline c$ but also on the dynamic rates $k$. 
Our aim is, as before for the Precision, to find estimates from above and below on the Sensitivity that depend on the stoichiometry but not on the parameters $\overline c$ and $k$.

\subsection{Upper bounds}

Our first result gives a very general upper bound on Sensitivity for all mass-action, detailed-balance systems.

\begin{theorem}[General upper bound on Sensitivity]
\label{th:upper-bound-sensitivity}
Given a kinetic system $(\sS,\sR,\sN,(\overline c,k))$, we have the Sensitivity bound
\[
S \leq \sqrt{\frac{\overline c_i}{\overline c_o}}.
\]
\end{theorem}

\begin{proof}
Let $u$ be the solution of~\eqref{eq:ode-linearized} with initial datum $e_i$.
Set $y_s = u_s\sqrt{\,\overline c_s}$. Then $y$ solves the equation
\begin{equation}
\label{def:ode-x}
\dot {y} = A_y y, \qquad (A_y)_{ss'} = -\frac1{\sqrt{\,\overline c_s}}\sum_{r\in\sR} \sN_{sr} k_r \sN_{s'r} \frac1{\sqrt{\,\overline c_{s'}}},
\qquad y(0) = {e_i}{\sqrt{\, \overline c_i}}.
\end{equation}
Since $A_y$ is symmetric and non-positive, it can be diagonalized with orthogonal matrices, $A_y = O^{-1} \Lambda O$, where $\Lambda =\diag(\lambda_1,\dots,\lambda _I)$ is a diagonal matrix of non-positive eigenvalues, and $O^T = O^{-1}$. Then $e^{tA_y} = O^{-1}e^{t\Lambda } O$, 
and we calculate, writing $o^k$ for the $k$-th column vector of $O$,
\def\ipc#1#2{(#1,#2)}
\begin{align*}
(e^{A_yt})_{ss'}& =  \ipc {e_s}{e^{A_yt}e_{s'}}= \ipc{e_s}{O^{-1}e^{\Lambda t} Oe_{s'}}= \ipc{Oe_s}{e^{\Lambda t} Oe_{s'}}= \ipc{o^s}{e^{\Lambda t}o^{s'}}.
\end{align*}
Since the eigenvalues are non-positive, this latter expression is bounded in absolute value by
\[
|o^s|\;\bigl|\,e^{\Lambda t}o^{s'}\bigr|
  = \Biggl(\sum_{\sigma\in \sS} e^{2t\lambda_\sigma}|o^{s'}_\sigma|^2\Biggr)^{1/2}
  \leq \Biggl(\sum_{\sigma\in \sS}  |o^{s'}_\sigma|^2\Biggr)^{1/2}
  = |o^{s'}| =1.
\]

By this calculation the solution $y$ of~\eqref{def:ode-x} satisfies
\[
\sup_{t\geq0} |y_o(t)| = {\sqrt{\,\overline c_i}}\,\sup_{t\geq0} \big| (e^{tA_y} e_i)_o\bigr| 
 = {\sqrt{\,\overline c_i}}\,\sup_{t\geq0} \big| (e^{tA_y})_{io}\bigr| 
 \leq {\sqrt{\,\overline c_i}}.
\]
The result then follows from transforming back from $y$ to $u$ and applying~\eqref{eq:alt-def-precision-and-sensitivity}.
\end{proof}

We will see in the examples below (e.g. Example~\ref{ex:sensitivity-larger-maxInvP}) that this bound is very far from being sharp.

\begin{remark}[Upper bound for normal systems]
If $A$ in~\eqref{eq:ode-linearized} is normal, it is orthogonally diagonalizable. Therefore by the same argument for $A_y$ in the previous theorem we have $S=(e^{tA})_{oi} \le 1$.
\end{remark}

\bigskip

Following the explicit formula for the Precision of homogeneous systems, one can also prove a property of the Sensitivity of homogeneous systems:
\begin{theorem}[Upper bound on Sensitivity for homogeneous  systems]
\label{thm:SensitivityUpBoundHomogenous}
If the system $(\sS,\sR,\sN,(\overline c,k))$ is homogeneous of some order $\kappa$, then $S\leq 1$.
\end{theorem}

\begin{proof}
We begin by introducing
\begin{equation*}
k_{ss'} :=
\begin{cases}
k_r \qquad \text{if $X_{s}$ and $X_{s'}$ react}, \\
0 \qquad \text{otherwise, }
\end{cases}
\end{equation*}
and we can assume without loss of generality that each reacting pair $s,s'$ is only connected by one reaction. 
Each column of $\sN$ has only two nonzero entries, which implies that the intersection of supports of two different rows of $\sN$ has at most one element. Thus the matrix $A$ in~\eqref{eq:ode-linearized} reads
\begin{equation*}
A_{ss'} =  -\dfrac{1}{\overline c_{s}}\sum_{r\in \sR} \sN_{sr}k_r  \sN_{s'r} =
\begin{cases}
-\dfrac{\kappa^2}{\overline c_{s}}\sum\limits_{\ell \in \sS} k_{s \ell }     &  s=s', \\
\dfrac{\kappa^2 k_{ss'} }{c_s} &  s\neq s'.
\end{cases} 
\end{equation*}
The matrix $A$ has negative diagonal and nonnegative off-diagonal entries. Based on this observation we proceed with the proof. Let $\mathbf{I}$ be the identity matrix, then for all $t\ge0$ there exists $\lambda <0$ such that $tA-\lambda  \mathbf{I}$ is element-wise nonnegative. We denote this property by $tA-\lambda  \mathbf{I} \succcurlyeq 0$. Powers of such a matrix preserve the property. The matrix exponential $e^{tA-\lambda  \mathbf{I}}$ is an infinite sum of elementwise nonnegative matrices, hence $e^{tA-\lambda  \mathbf{I}}\succcurlyeq 0$. On the other hand the two matrices $\lambda \mathbf{I}$ and $tA-\lambda  \mathbf{I}$ commute. By the properties of matrix exponentials we obtain
\[
e^{tA}  = e^{\lambda \mathbf I}e^{tA-\lambda  \mathbf{I}} = e^{\lambda} e^{tA-\lambda  \mathbf{I}} \succcurlyeq 0. 
\]
The matrix $A$ has the property that the sum of the entries of each row is zero. Let $\mathbf{1} = (1,\hdots ,1)^T$, then $A \mathbf{1} =0$ which implies that $e^{tA} \mathbf{1}=\mathbf{1}$. Each row sum of nonnegative entries of $e^{tA}$ is 1, therefore $ (e^{tA} )_{ss'} \le 1$ for all $s,s'$. The alternative formulation of Sensitivity~\eqref{eq:alt-def-precision-and-sensitivity} then completes the proof:
\[
S =  \sup_{t\geq0} (e^{tA})_{oi} \le 1.
\] 
\end{proof}

The value $S=1$ is special, for the following reason. In some cases one can concatenate, or `daisy-chain' systems, by feeding the output of one system into the input of another. We conjecture that the sensitivity of the chain can never exceed the product of the sensitivities of the individual components. If this is true, then the value $S=1$ is critical; it only makes sense to daisy-chain components with $S>1$. Whether the conjecture is true or not, for some systems tuning of the parameters allows one to approximately achieve product Sensitivity, as the next example shows. 

\begin{example}
\refstepcounter{example}
\label{ex:Daisy}
\noindent
\textbf{Example \ref{ex:Daisy}} (Daisy chaining) 
We extend Example A to the following set of reactions:
\begin{align}
 2 X_1  \overset{k_1 }{\underset{ }\leftrightharpoons }   X_2 + X_3,    \qquad 
 2 X_3   \overset{k_2 }{\underset{ }\leftrightharpoons }  X_4  ,         \qquad
2 X_4  \overset{k_3}{\underset{ }\leftrightharpoons }    X_5   ,
\qquad             
 X_1   \overset{k_4}{\underset{ }\leftrightharpoons }  X_2      .         \nonumber
\end{align} 
The parameters $\overline{c}$ and $k$ are chosen to give rise to four separate timescales, as shown in Figure~\ref{fig:Daisy-Chaining}. A perturbation in input $X_1$ in the first reaction, which is the fastest, results in a quick rise in~$X_3$. In the second reaction the species $X_3$ behaves like an input and amplifies $X_4$. Species $X_4$ shows a Sensitivity near 2 relative to $X_3$, and near $2\cdot2=4$ relative to $X_1$. This chaining is further extended by feeding $X_4$ to $X_5$, obtaining a Sensitivity close to $2\cdot2\cdot2=8$ for $X_5$ relative to input $X_1$. The final reaction is the slowest one, and $X_2$ acts as a buffer whose concentration is considerably larger than the other species. At the slowest time scale, the last reaction reduces the initial rise in $X_1$ and pushes back all the species to a concentration very close to pre-stimulus level, thus creating a large Precision. 
\begin{figure}[H]
\begin{center}
\labellist
\small
\pinlabel{time} [t] at 245 2
\pinlabel{\rotatebox{90}{concentration}} [r] at -8 200
\endlabellist
  \includegraphics[width=.4\textwidth]{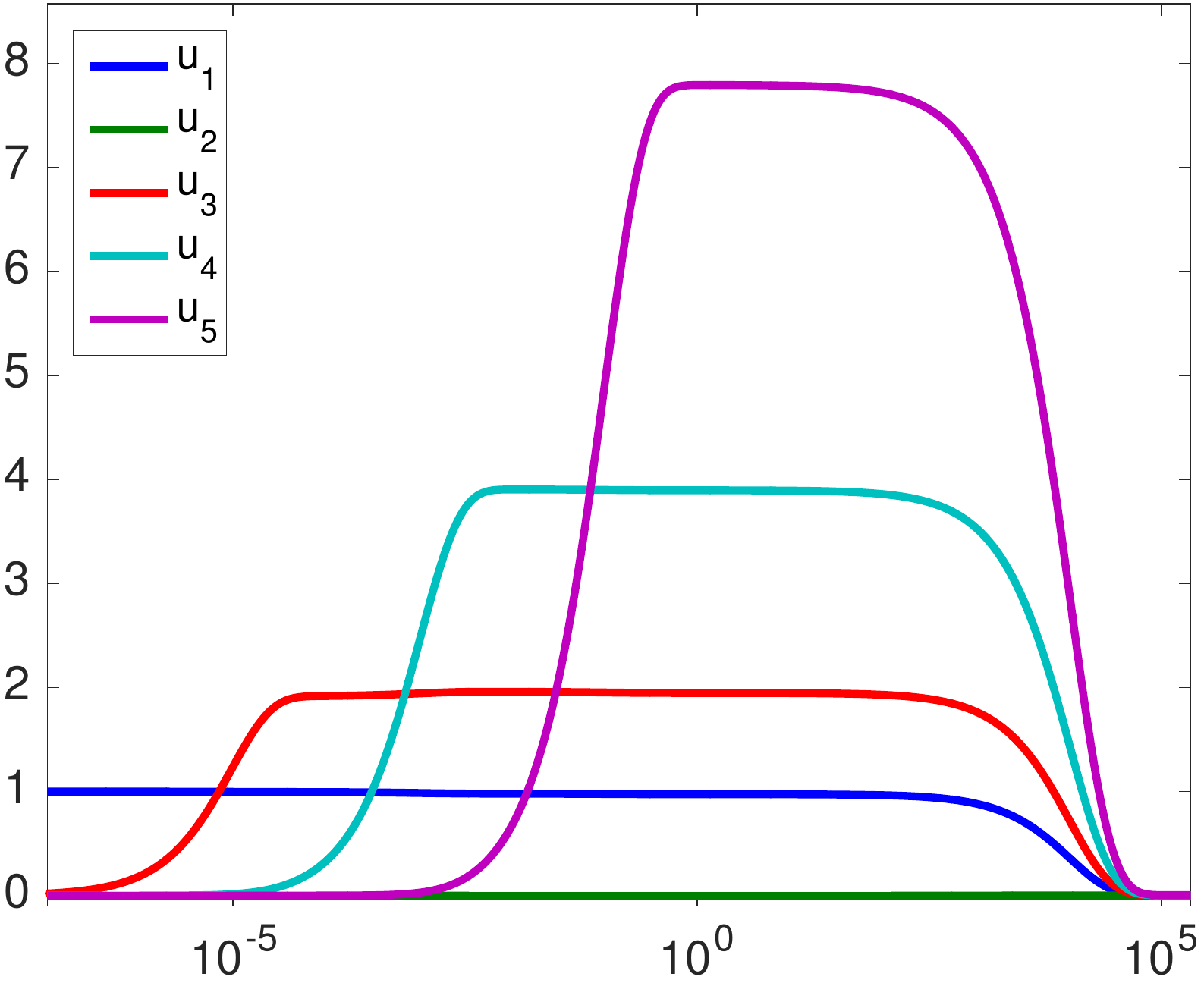} 
\end{center}
\caption{An example of routing the output of one system into the input of a next system (`daisy-chaining'), which in this case allows one to achieve a total Sensitivity close to the product of the individual sensitivities. Here $\overline c = (10,10.000, 0.01, 0.01, 0.001)$ and $k = (1000, 10, 0.01, 0.001)$. }
\label{fig:Daisy-Chaining}
\end{figure}
\end{example}
\subsection{Intermezzo: subsystems}

In Theorem~\ref{th:PSsubsystems} below we construct lower bounds for the maximal Sensitivity by using properties of a \emph{subsystem} and exploiting the possibility of making the subsystem dynamics much faster than the dynamics in the remainder of the system. We first study the relation between optimal Precision and Sensitivity of a subsystem with that of the full system.

\begin{definition}[Subsystems]
\label{def:subsystems}
$(\sS,\sR_1,\sN_1)$ is a \emph{subsystem} of $(\sS,\sR,\sN)$, notation  $(\sS,\sR_1,\sN_1)\subset (\sS,\sR,\sN)$, if $\sR_1\subset \sR$ and $\sN_1$ is the restriction of $\sN$ to the columns given by $\sR_1$. 
\end{definition}

One can obtain a subsystem by setting the rates $k_r$ of some of the reactions to zero. Note that this is different from setting them to \emph{nearly} zero, since the stoichiometric freedom $\range (\sN)$ is different in the two cases, and therefore the stationary states are also different.

\begin{theorem}[Precision and Sensitivity under taking subsystems]
\label{th:PSsubsystems}
If $(\sS,\sR_1,\sN_1)$ is a \emph{subsystem} of $(\sS,\sR,\sN)$, then 
\begin{itemize}
\item The maximal Sensitivity of the subsystem is less than or equal to the maximal Sensitivity of the full system, and 
\item The maximal and minimal Precision of the subsystem may be smaller than, equal to, or larger than in the full system.
\end{itemize}
\end{theorem}

\begin{proof}
For the purposes of Precision and Sensitivity, the two systems (the `system' and the `subsystem')  are both described by equations of the form~\eqref{eq:ode-linearized}. We can take the same set $\sR$ of reactions for both, if for the duration of this proof we allow some of the parameters $k_r$ for the subsystem to be zero. 

The behaviour of the Sensitivity now follows from the continuity properties of ordinary differential equations. We  write the solution of~\eqref{eq:ode-linearized} with initial datum $e_i$ as $u(t;\overline c,k)$ to emphasize the choice of parameters. For each $\e>0$, we can find a parameter point $(\overline c^1,k^1)$ for the subsystem (which implies that some of the $k_r^1$ are zero), and a time $t^1\geq0$, such that 
\[
u(t^1;\overline c^1,k^1) \geq \maxS(\sS,\sR_1,\sN_1)  -\e.
\]
On the finite time interval $[0,t^1+1]$, solutions depend continuously on parameters, implying that we can find strictly positive parameter points $(\overline c^2,k^2)$ for the full system such that 
\[
\sup_{t\in[0,t^1+1]} u(t;\overline c^2,k^2) \geq \maxS(\sS,\sR_1,\sN_1)  -2\e.
\]
Since $\e>0$ is arbitrary, it follows that 
\[
\maxS(\sS,\sR,\sN)\geq \maxS(\sS,\sR_1,\sN_1).
\]

A similar argument fails for the Precision, since the two limits $t\to\infty$ and $k_r\downarrow 0$ need not commute. As examples where the Precision of a system is larger or smaller than the Precision of a subsystem, consider
\begin{itemize}
\item If we choose a system with finite Precision and a subsystem in which the input and output species are no longer connected by any reactions, then the output concentration is independent of the input concentration, implying an infinite Precision, which is therefore larger than the Precision of the full system. 
\item In Example~\ref{ex:Daisy} the Precision of the full system is high, while some of the subsystems have low Precision. 
\end{itemize}
\end{proof}

\subsection{Properties of the Sensitivity: lower bounds}
\label{subsec:maxS-maxInvP}

Since $\sup_{t\geq0} c_o^\e(t) \geq c_o^\e(+\infty)$, the inequality $SP\geq1$ always holds (compare the definitions of Sensitivity~\eqref{defeq:sensitivity} and Precision~\eqref{defeq:precision}) and therefore $\maxS\geq \maxInvP$. The next theorem strengthens this property.

\begin{theorem}[Maximal Sensitivity is bounded from below by the maximal inverse Precision over all subsystems]
Given a system $(\sS,\sR,\sN)$, 
\[
\maxS(\sS,\sR,\sN) \geq \max \Bigl\{ \maxInvP(\sS,\sR_1,\sN_1): (\sS,\sR_1,\sN_1)\subset (\sS,\sR,\sN)\Bigr\}.
\]
\end{theorem}

\begin{proof}
The proof of this theorem is very similar to that of Theorem~\ref{th:PSsubsystems}.  For any $(\sS,\sR_1,\sN_1)\subset (\sS,\sR,\sN)$ and $\e>0$, we choose a parameter point $(\overline c,k^1)$ for $(\sS,\sR_1,\sN_1)$ (i.e. with $k_r^1=0$ whenever $r\in\sR\setminus \sR_1$) such that 
\[
P^{-1}(\sS,\sR,\sN,(\overline c,k^1)) = P^{-1}(\sS,\sR_1,\sN_1,(\overline c,k^1)) \geq  \maxInvP(\sS,\sR_1,\sN_1)- \e.
\]
We then choose $t^1>0$ such that 
\[
u(t^1;\overline c,k^1) \geq \maxInvP(\sS,\sR_1,\sN_1) - 2\e.
\]
Finally, using continuous dependence on parameters we choose a strictly positive parameter point~$k^2$ such that 
\[
u(t^1;\overline c,k^2) \geq \maxInvP(\sS,\sR_1,\sN_1)- 3\e.
\]
Therefore
\[
\maxS(\sS,\sR,\sN) \geq \maxInvP(\sS,\sR_1,\sN_1) - 3\e,
\]
and since $\e>0$ and the subsystem $(\sS,\sR_1,\sN_1)$ were arbitrary, the result follows.
\end{proof}

\begin{example}
\refstepcounter{example}
\label{ex:sensitivity-larger-maxInvP}
\noindent
\textbf{Example \ref{ex:sensitivity-larger-maxInvP}} (Sensitivity larger than maxInvP).
We consider a chemical reaction network with~6 species and three reactions:
\begin{align*}
X_1+X_3   \overset{k_1}{\underset{ }\leftrightharpoons }  X_6,  \qquad
X_1+X_4   \overset{k_3}{\underset{ }\leftrightharpoons }  X_5,  \qquad
X_2+X_6   \overset{k_2}{\underset{ }\leftrightharpoons }  X_5.          
\end{align*}
Choosing $X_1$ to be the input and $X_6$ the output, Theorem~\ref{th:UpperBoundsInvP} gives $\maxInvP=1$.
Figure~\ref{fig:overshoot}, however, shows a value for the sensitivity of about $1.13$, which therefore exceeds $\maxInvP$. A formal argument inspired by a numerical observation suggests that the Sensitivity is bounded from above by $1+e^{-2}\approx 1.135$, but proving this bound remains open.
\begin{figure}[H]
\labellist
\small
\pinlabel{time} [t] at 255 2
\pinlabel{\rotatebox{90}{concentration}} [r] at -7 190
\endlabellist
\includegraphics[height=0.35\textwidth]{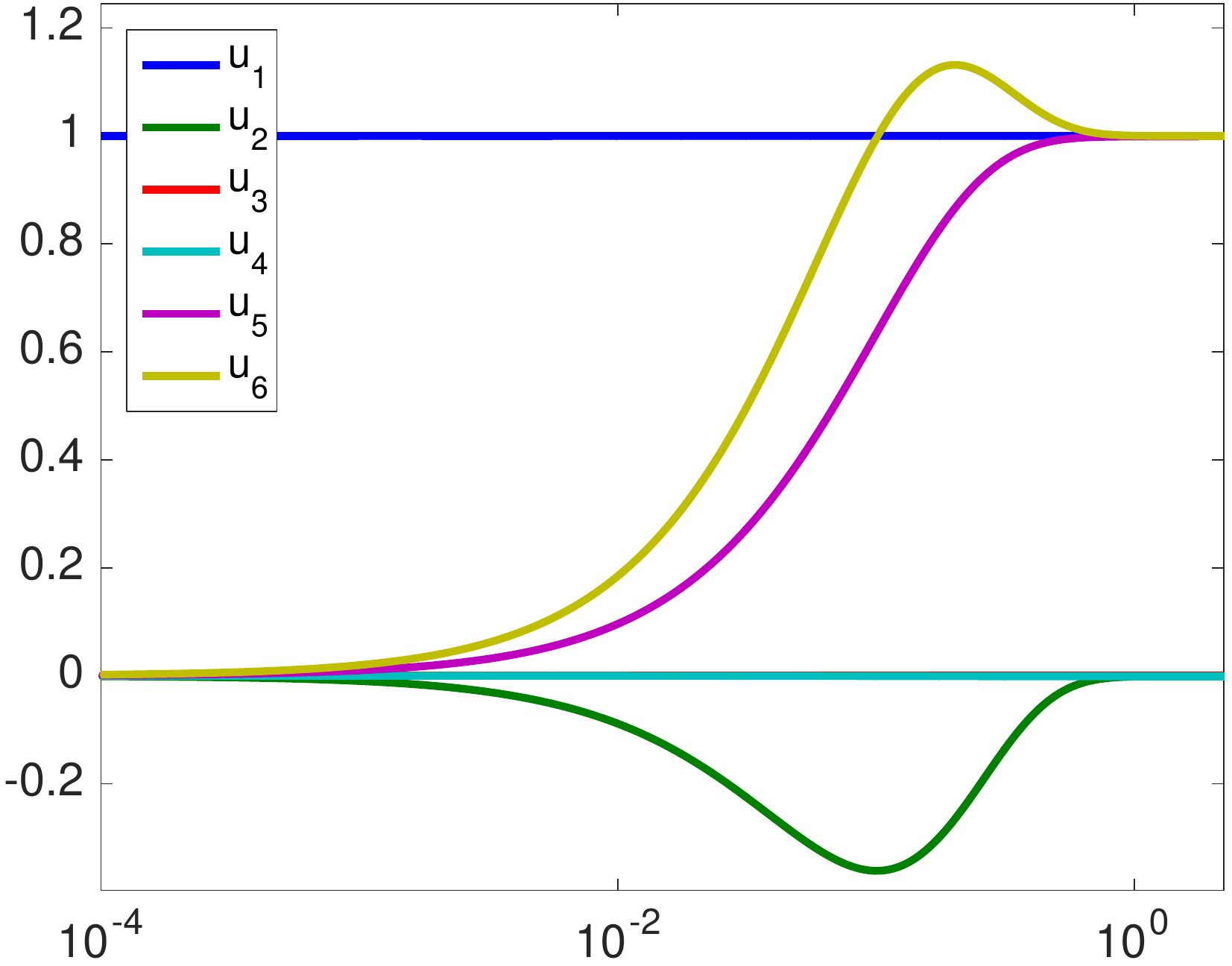}
\vskip2\jot
\caption{ Solution of $\dot{u}=Au$ with $u(0)=e_1$ and parameters $k=(0.1,10,10^4)$ and $\overline{c}=(10^4,10^{-2},10^4,10^3,1,10^{-4} )$. Here  $S\approx1.13$. Note that this number is much smaller than $\sqrt{\overline c_i/\overline c_o} = 10^4$.   }
\label{fig:overshoot}
\end{figure}
\end{example}

\section{\label{sec:matroid}Matroid theory and the proof of Theorem~\ref{th:UpperBoundsInvP}}
\subsection{Matroids and the combinatorics of the stoichiometric space}
The proof of Theorem~\ref{th:UpperBoundsInvP} was conceived with a certain matroid related to the stoichiometric matrix $\sN$ in mind. Seeing that the entire argument could be also stated in terms of linear algebra, we chose to avoid the use of this concept in our presentation of the proof. We will give the matroid perspective here as an optional service to the reader.  We briefly describe the relevant matroid theory here, referring to the book of Oxley \cite{Oxley14} for a more detailed account and full proofs of the statements below. 

Consider a finite set of vectors $E\subseteq \R^n$ and let 
$$\mathcal I:=\{ X\subseteq E: X\text{ is linearly independent over }\R\}.$$
Then $\mathcal I$ has the following three properties.
\begin{itemize}
\item[\bf(I0)] $\emptyset\in \mathcal I$
\item[\bf(I1)] if $X\in \mathcal I$ and $Y\subseteq X$, then $Y\in \mathcal I$
\item[\bf(I2)] if $X, Y\in \mathcal I$ and $|X|<|Y|$, then there exists an $e\in Y\setminus X$ so that $X\cup\{e\}\in \mathcal I$.
\end{itemize}
A {\em matroid} is any pair $M=(E,\mathcal I)$ where $E$ is a finite set and $\mathcal I$ is a set of subsets of $E$ satisfying the above three axioms. In this more abstract setting, we also call a set $X\subseteq E$ {\em independent} if $X\in\mathcal I$ and {\em dependent} otherwise. A set $B\subseteq E$ is called a {\em basis} if $B$ is an inclusion-wise maximal independent set, and $C\subseteq E$ is a {\em circuit} if $C$ is an inclusion-wise minimal dependent set. 

Let $M=(E, \mathcal I)$ be a matroid, and let ${\mathcal B}$ be the set of bases of $M$. Then the set 
$${\mathcal B}^*:=\{ E\setminus B: B\in \mathcal B\}$$
is the set of bases of another matroid $M^*$, the {\em dual} of $M$. 

Given this elementary result in matroid theory, the proof of the following statement is straightforward.
\begin{lemma}\label{lem:matroid_basic}If $C_0$ is a circuit of $M$ and $e_0\in C_0$, then there is a basis $B$ of $M$ such that
\begin{itemize}
\item $C_0\setminus \{e_0\}\subseteq B$, and $e_0\not\in B$;
\item for each $e \in B$, there is a circuit $D$ of $M^*$ such that $D\subseteq (E\setminus B)\cup \{e\}$; and
\item for each $e \in E\setminus B$, there is a circuit $C$ of $M$ such that $C\subseteq B\cup \{e\}$.
\end{itemize}
\end{lemma}
\proof As $C_0$ is a circuit of $M$, the set $I:=C_0\setminus\{e_0\}$ is an independent set. Let $B$ be any inclusion-wise maximal set containing $I$. Then $B$ is a basis of $M$. Since $B$ is independent, the circuit $C$ cannot be fully contained in $B$, so $e_0\not\in B$. Since $B$ is maximal, the set $B\cup \{e\}$ is dependent for each $e\in E\setminus B$, and hence contains a circuit $C$ of $M$. By definition of the dual, $E\setminus B$ is a basis of $M^*$. Since $E\setminus B$ is a maximal independent set of $M^*$, the set $(E\setminus B)\cup \{e\}$ is dependent in $M^*$ for each $e\in B$, hence contains a circuit $D$ of $M^*$. \endproof

Given any $r\times E$ matrix $A$, define
$${\mathcal I}_A:=\{X\subseteq E: \text{the colums of }A\text{ indexed by $X$ are linearly independent}\}.$$
Then $M(A):=(E, {\mathcal I}_A)$ is a {\em linear matroid}. 

\begin{lemma} Let $A$ be an $r\times E$ matrix, and let $M=M(A)$. Then $C$ is a circuit of $M$ if and only if $C=\supp(x)$ for an elementary vector $x\in \ker(A)$.
\end{lemma}
\proof A set $X\subseteq E$ is dependent in $M$ if and only if there is a linear dependency among the columns of $A$ pointed out by $X$, i.e. a nonzero vector $x\in \R^E$ with $Ax=0$ and $\supp(x)\subseteq X$. \endproof
The circuits of  the dual of $M(A)$ can be similarly characterized.
\begin{lemma} Let $A$ be an $r\times E$ matrix, and let $M=M(A)$. Then $D$ is a circuit of $M^*$ if and only if $D=\supp(y)$ for an elementary vector $x\in \rowspace(A)$.
\end{lemma}

Now the stoichiometric matrix $\sN$ is an $\sS\times \sR$ matrix, and thus the transpose matrix $\sN^T$ is an $\sR\times \sS$ matrix. The matroid $M=M(\sN^T)$ has ground set $\sS$ and divides the subsets of $\sS$ in dependent and independent sets. With $W:=\range(\sN)$, the circuits of $M$ are the minimal supports of vectors $u\in \ker(\sN^T)=W^\perp$, and the circuits of $M^*$ are the minimal supports of vectors $w\in \rowspace(\sN^T)=\range(\sN)=W$. 

Lemma \ref{lem:basic} follows directly by applying Lemma \ref{lem:matroid_basic} to $M=M(\sN^T)$, $C_0=\supp(u^*)$, and $e_0=i$.

\subsection{Computing the upper bound on the inverse precision}
Let $C$ be a circuit of the stoichiometric matroid $M=M(\sN^T)$. Finding a vector $u\in W^\perp$ with $C=\supp(u)$ and $u_i=1$ is a matter of elementary linear algebra. Computing the maximum
$$\max \{ u_o:  u\text{ elementary in } W^\perp,\  u_i=1, \text{ or } u=0\}$$
reduces to enumerating the collection of circuits of the stoichiometric matroid. In the same vein, to determine 
$$\max \{ w_i:  w\text{ elementary in } W,\  w_o=-1, \text{ or } w=0\},$$
it suffices to enumerate the circuits of $M^*$.

The number of circuits of a matroid on $n$ elements can be exponential in $n$, and so we cannot expect to enumerate the full set of circuits in polynomial time. Boros et al.~\cite{Boros2003} describe a simple algorithm which will enumerate the circuits of a matroid  in {\em incremental polynomial time}. That is, there exists a polynomial $p(n,k)$ so that listing the first $k$ circuits of a matroid $M=(E, \mathcal I)$ takes their algorithm $p(|E|, k)$ time. In a related paper~\cite{Khachian2005},  an algorithm is described which will  generate the circuits containing a fixed element in incremental polynomial time. For our application, we would like to enumerate the circuits containing two fixed elements of the ground set, but it appears to be an open problem whether this can be done in incremental polynomial time.  On the practical side, SAGE, the open-source computer algebra system, implements several algorithms for enumerating the circuits of a matroid.

Given that there are two ways to determine the upper bound, whose running times will depend on the number of circuits of $M$ or $M^*$, one would like to estimate which one of these matroids has the least number of circuits. 

It is straightforward that in in a matroid, any two bases have the same cardinality. The {\em rank} of a matroid $M$ is the cardinality of any basis of $M$. A matroid of rank $r$ on $n$ elements may have as many as $\binom{n}{r+1}$ circuits, the maximum being attained by the {\em uniform} matroid of rank $r$. 

The rank of the stoichiometric matrix $M(\sN^T)$ equals $r=\rank(\sN)$, and the rank of its dual  $M(\sN^T)^*$ is $n-r$, where $n=|\sS|$ is the size of the ground set. Taking the maximum number of circuits of a matroid of rank $r$ as a coarse estimate for the true number of circuits, we expect that  in general $M(\sN^T)$ will have fewer circuits than $M(\sN^T)^*$ while $2r\leq n$.

\section{Non-detailed-balance chemical reaction networks}
\label{sec:dissipative}
 
We now briefly comment on systems with mass-action kinetics but without the detailed-balance assumption. In these systems the kinetic function $\sK$ has the form
\begin{equation}
\label{eq:mass-action}
\sK_r(c) =  k_r^+ c^{\alpha_r} - k_r^- c^{\beta_r} ,
\end{equation}
where  $\alpha_r$ and $\beta_r$ are as in Definition~\ref{def:system}, and $k_r^+$ and $k_r^-$ are arbitrary non-negative coefficients. The network is called \emph{reversible} if $k_r^\pm>0$.

In the case of detailed-balance systems, we chose to perturb the system by adding a small amount of a certain species. Although in non-detailed-balance systems there are more choices for perturbation, here we stick to the same method. Definitions~\ref{def:sensitivity} and~\ref{def:precision} for Sensitivity and Precision do not rely on the assumption of detailed balance, whereas in the alternative formulations~\eqref{eq:alt-def-precision-and-sensitivity}, the matrix $A$ appears and this matrix owes its structure to the detailed-balance assumption. First we provide an alternative formulation for non-detailed-balance systems.
 \begin{lemma}[Alternative formulations of Precision and Sensitivity in non-detailed-balance systems]
 Let $t\mapsto u(t)$ be the solution of $\dot{u} = \hat{A} u$ with 
 \[ \hat{A}_{ss'} =  -\frac{1}{\overline c_{s}}\sum_{r\in \sR} (\alpha_{sr} - \beta_{sr} )   (\alpha_{s'r}  k_r^+ \overline{c}^{\alpha_r}   - \beta_{s'r} k_r^- \overline{c}^{\beta_r}   ),   \]
 and initial data $u(0)=e_i$. 
 The Precision and Sensitivity then have the alternative formulations
 \begin{equation}
 \label{eq:NDB-alt-def-precision-and-sensitivity}
 P^{-1} = \lim_{t\to\infty} u_o(t) = \lim_{t\to\infty} (e^{t\hat{A}})_{oi}
 \qquad
 \text{and}
 \qquad
 S =  \sup_{t\geq0} u_o(t) = \sup_{t\geq0} (e^{t\hat{A}})_{oi}.
 \end{equation}
 \end{lemma}
\begin{proof}
The proof is again a simple manipulation.
\end{proof}
  One can ask what happens with the bounds on Precision and Sensitivity when detailed balance does not hold. We start by obtaining a bound for reversible unimolecular reaction networks that do not necessarily satisfy detailed balance.
\begin{theorem}
\label{th:upper-bound-sensitivity-dissipative}
 In a reversible unimolecular reaction network we have $S\le 1$.
\end{theorem}
\begin{proof}
Reactions in such a network are of the type
  \[X_s  \overset{k_{s's} }{\underset{k_{ss'} }\rightleftharpoons }  X_{s'} .\]
with both rate constants $k_{s's}$ and $k_{ss'}$ strictly positive when $X_s$ and $X_{s'}$ react with each other and we assume them to be zero otherwise. This allows to write the following ODE for the evolution of each species.
\begin{equation*}
\dot{c}_s = \sum_{\ell \in \mathcal{S}} k_{s\ell} c_{\ell } - \sum_{\ell \in \mathcal{S}} k_{\ell s} c_s.  \nonumber
\end{equation*}
Let $u_s(t) = (c_s(t) - \overline{c}_s)/\overline{c}_s$, then $u$ solves $\dot{u} = \hat{A} u$ where
\begin{equation*}
\hat{A}_{ss'} =  
\begin{cases}
    -  \sum_{\ell} k_{\ell s}  &  s=s', \\\\
    \frac{\overline{c}_{s'}}{\overline{c}_s} k_{ss'}   &  s\neq s'.
\end{cases} 
\end{equation*}
We note that $\hat{A}$ has negative diagonal and nonnegative off-diagonal elements, moreover $\hat{A} \mathbf{1} =0$. Showing that $S=(e^{t\hat{A}})_{oi}\le 1$ is similar to the argument in Theorem~\ref{thm:SensitivityUpBoundHomogenous}.
 \end{proof}
 
Consider the following example of how a small system can achieve large Sensitivity and Precision.

\begin{example}
\refstepcounter{example}
\label{ex:dissipative}
\noindent \textbf{Example \ref{ex:dissipative}} (Adaptation in non-detailed-balanced systems).
Consider a receptor~$R$, a ligand $L$, a phosphate group $p$, complexes $Rp$, $RL$, $RLp$, and $Y$ (all indexed from 1 to 7 respectively) that participate in reactions depicted in Figure~\ref{fig:ExampleRLp}. Let $L$ serve as input and $RLp$ output of the network. Figure~\ref{fig:ExampleRLp} shows how in the absence of detailed balance a Sensitivity near 70 can be achieved. Note that the detailed-balance version of the network has $\maxInvP =1$. In fact, if we omit the last reaction, then with the same parameters an inverse Precision near 70 is achieved. The last reaction acts as a feedback with delay and performs the adaptation step. One can further increase the Precision by increasing the concentration of $Y$ and making sure that the last reaction is the slowest one. 
\begin{figure}[H]
\begin{center}
\labellist
\small
\pinlabel{time} [t] at 245 2
\pinlabel{\rotatebox{90}{concentration}} [r] at -5 180
\endlabellist
  \includegraphics[width=.4\textwidth]{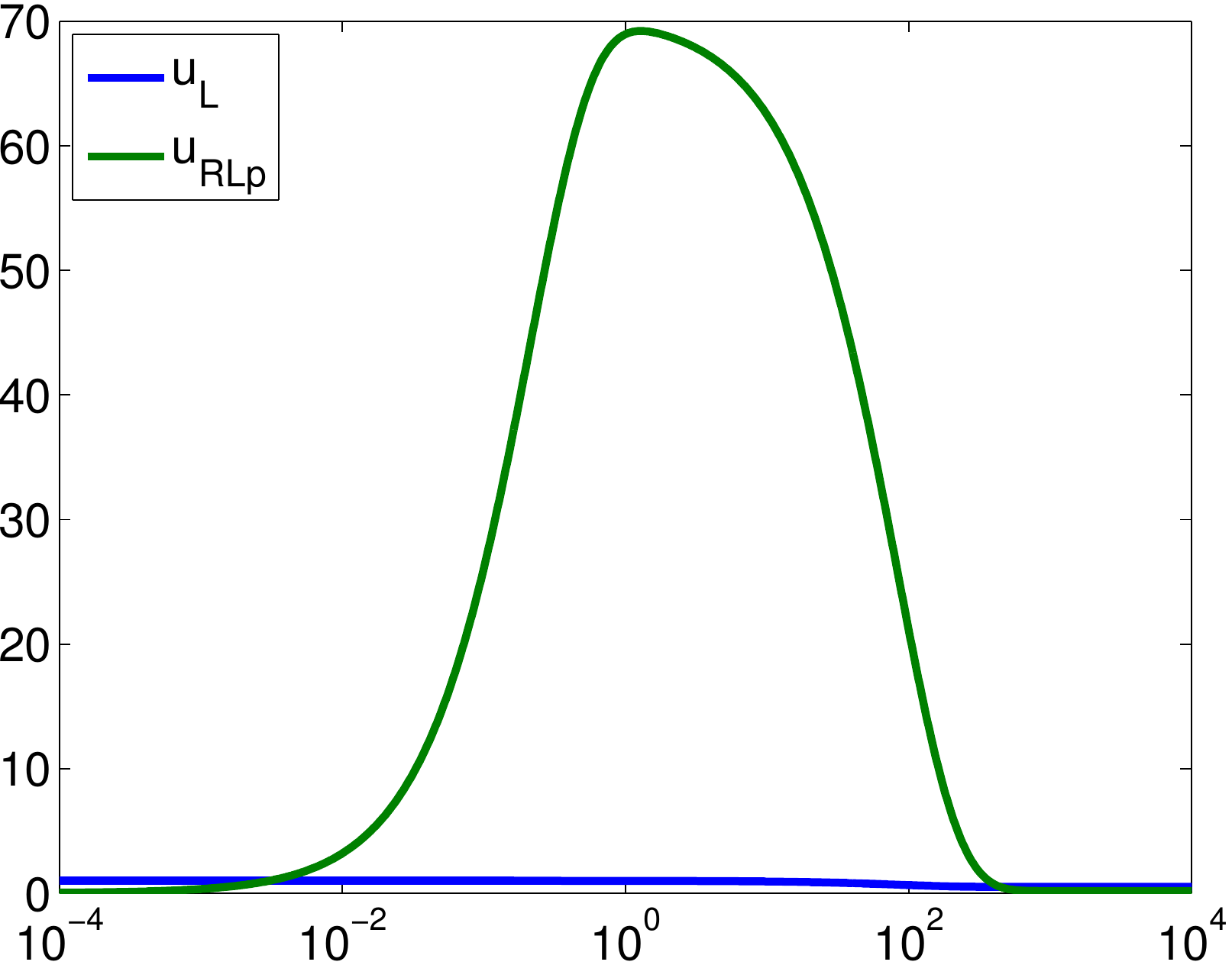} 
  \qquad
    \includegraphics[width=.4\textwidth]{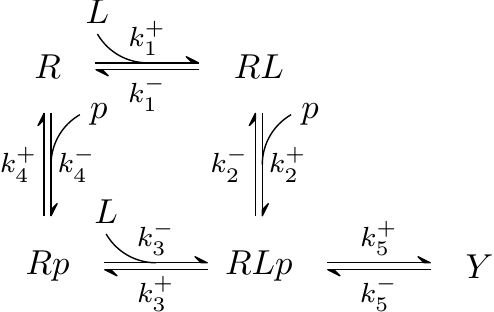} 
\end{center}
\caption{ Plot of $u_L(t)$ and $u_{RLp}(t)$ obtained from solving $\dot{u} = \hat{A} u$ for $u(0) =e_2$. Rate constants are $k^+=(10^{-3},10^4,10^{-3} ,10^{-3}, 1)$ and $k^-=(17.7, 10^4, 10^{-5}, 10^{5},  3\cdot10^{-5})$ and the steady state $\overline{c} = (0.03, 3481.5,  0.042,   3481.5,   6.85,   0.3, 10^4)^T$.}
\label{fig:ExampleRLp}
\end{figure}

\end{example}

\section{Summary and Discussion}
\label{sec:discussion}

\subsection{Summary}
The analysis of this paper is sparked by the question we posed in the Introduction, \emph{To which extent can a {non-dissipative} system perform  adaptation?} 
We investigated this question by first defining `non-dissipative' as `detailed-balance, mass-action' and `performing adaptation' as `having high Sensitivity and Precision', and then deriving a number of rigorous results about such systems.

Concretely, we prove that 
\begin{enumerate}
\item \emph{Unimolecular} reactions can have high Precision (Theorem~\ref{th:precision-homogeneous-systems}) but their Sensitivity is bounded by one (Theorem~\ref{thm:SensitivityUpBoundHomogenous}), even if the detailed-balance restriction is relaxed (Theorem~\ref{th:upper-bound-sensitivity-dissipative});
\item The {maximal} inverse Precision of a given system can be characterized in various combinatorial ways (Theorem~\ref{th:UpperBoundsInvP});
\item The maximal Sensitivity of a given system is at least as large as the maximal inverse Precision over all subsystems (Theorem~\ref{th:PSsubsystems}) and can sometimes be larger (Example~\ref{ex:sensitivity-larger-maxInvP}).
\item By modifying not only the coefficients but also the stoichiometry, Precision and Sensitivity can be made arbitrarily large (Examples~\ref{ex:B} and~\ref{ex:Daisy}).
\end{enumerate} 

In this way we show that \emph{non-dissipative systems can be arbitrarily adaptive.} 
This does require `extreme' systems however, in the sense of having large stoichiometry, large concentration ratios, and/or large time scale ratios. Theorem~\ref{th:upper-bound-sensitivity} shows that large concentration ratios are necessary for large Sensitivity. For the other two we have no rigorous characterization, but the examples suggest that at least large stoichiometry (Example~\ref{ex:B}) or large time scale ratios (Example~\ref{ex:Daisy}) are necessary for good performance. 

\subsection{Discussion}
We now comment on a number of aspects of this work. 

\medskip

\emph{Definition of `non-dissipative' systems.} Detailed-balance, mass-action systems are a natural choice for `non-dissipative' systems. They can be considered thermodynamically closed, and admit a free-energy functional $F$ that drives the evolution in a gradient-flow structure~\cite{GlitzkyMielke13,MaasMielkeInPrep}. In this context, one can identify `dissipation' with the instantaneous decrease of $F$, and the system is therefore non-dissipative in the sense that at all stationary points $F$ is constant. 

Despite these nice properties, this family does contain some weird specimens, such as
\[
X_1 \leftrightharpoons 2X_2, \qquad X_2 \leftrightharpoons 2X_1,
\]
for which the stoichiometric subspace is the \emph{whole} space of positive concentrations, and which clearly can not be mass-conservative in the traditional sense. In our examples we avoided such exotic species, and concentrated on systems that can be realized with actual chemical systems.

\medskip
\emph{Relation between `adaptation' and Precision and Sensitivity.}
In this paper we focus on Precision and Sensitivity as proxies for a more elaborate concept of adaptation. A better concept of adaptation might include (a) the persistence of `good' behaviour across a range of input concentrations, allowing for continuous tracking in the direction of increasing concentration, and (b) a measure of `temporary response' that measures not the instantaneous maximum of a concentration (like our Sensitivity) but some quality of a downstream machinery that acts on this concentration.

\medskip
\emph{Role of matroid theory.} 
The proof of the characterization of maximal  inverse Precision, Theorem~\ref{th:UpperBoundsInvP}, is formulated in linear-algebra terminology, but in fact the ideas are inspired by matroid theory, as explained in Section~\ref{sec:matroid}. Matroid theory arises in this context through the maximization over positive coefficients $k$ and $\overline c$, by which equality constraints become replaced by sign constraints (Lemma~\ref{lem:o_u} is a good illustration of this). The framework of matroid theory provides a natural structure in which to connect different characterizations of the same object, as illustrated by Theorem~\ref{th:UpperBoundsInvP}, and for this reason has been used in other works on chemical reaction networks~\cite{BeardBabsonCurtisQian04,MullerRegensburgerSteuer14,Reimers14TH}.

\medskip

\emph{Role of definitions.}
The conclusion of this paper, that non-dissipative systems can perform arbitrarily effective adaptation, serves as an illustration that the relationship between dissipation and functionality that is often broadly claimed in the literature requires very careful consideration; precise definitions are necessary, and at this stage it is not quite clear how to best choose these definitions, in order to obtain the clearest statements and most useful insight. 

\subsection{Comparison with~\cite{LanSartoriNeumannSourjikTu12,LanTu13}}
The results of this paper appear to be in contradiction with  remarks by Lan, Tu, and co-authors~\cite{LanSartoriNeumannSourjikTu12,LanTu13} that e.g.\ `adaptation is necessarily a non-equilibrium process and it always costs (dissipates) energy'~\cite{LanSartoriNeumannSourjikTu12} or `the I1-FFL (Incoherent type-1 feed-forward loop) network always operates out of equilibrium'~\cite{LanTu13}.

The discrepancy stems from a difference in definitions: both papers assume  a type of feedback that only exists in non-equilibrium systems.
Consider, as an example, the simple reaction $A\leftrightharpoons B$. From one point of view, this reaction encodes only \emph{positive} influence of $A$ on $B$ and vice versa, since starting from equilibrium, increasing $A$ leads to increase in $B$. From this point of view, a negative feedback mechanism can not be built using equilibrium building blocks, since negative feedback would require a negative influence. The systems of the present paper therefore fall outside of the scope of~\cite{LanSartoriNeumannSourjikTu12,LanTu13}.

However, in this simple reaction one can also observe negative influence, through the mechanism of \emph{redistribution}. Consider for instance  the following quantitative version:
\[
\dot A = B-10A  , \qquad \dot B = 10A - B.
\]
(This corresponds to $\overline c_A = 0.1$, $\overline c_B = 1$, and $k=1$ in the setup of this paper). If, starting from equilibrium $A=0.1$, $B=1$, we increase both $A$ and $B$ by the same amount, then the reaction will redistribute the total additional amount in the ratio $10:1$, as illustrated in Figure~\ref{fig:LanTu}. This has the same qualitative effect as negative feedback of $B$ on $A$ would have, as illustrated by the figure.
\begin{figure}[H]
\labellist
\small
\pinlabel{time} [t] at 245 2
\pinlabel{\rotatebox{90}{concentration}} [r] at -7 90
\pinlabel{\rotatebox{90}{concentration}} [r] at -7 300
\endlabellist
\includegraphics[height=6cm]{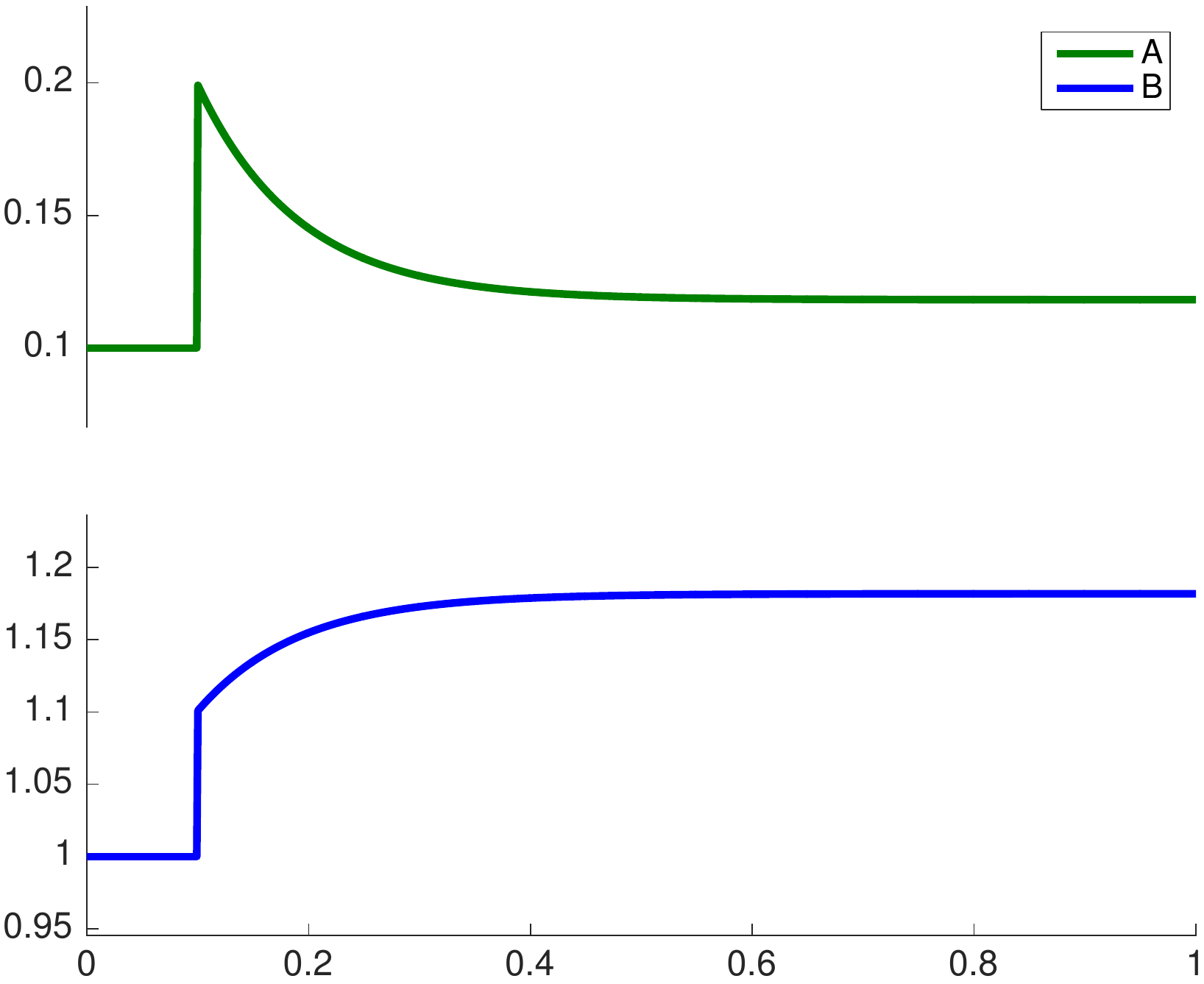}
\caption{Effective negative feedback in the reaction $A\leftrightharpoons B$.}
\label{fig:LanTu}
\end{figure}

This simple example allows us to explain how the non-dissipative systems of this paper have an effect very similar to the incoherent type-1 feed-forward loop (I1-FFL) studied in~\cite{LanTu13}. Consider the following two systems:

\begin{figure}[H]
\subfigure[]
{\label{subfig:LT}
\schemestart
$X_1$ \arrow(x1--x3){->}[-90,2] $X_3$
\arrow(@x1--x2){->}[-50,1.2] $X_2$ \arrow(--x3){-|}[-130,1.2] 
\schemestop}
\qquad\qquad\qquad
\subfigure[]{\label{subfig:ExA}
\schemestart
$2X_1$ \arrow(x1--x3){s>[+(-70:1.3)]}[-90,2] $X_3$
\arrow(@x1--x2){s<>[+(-70:1.3)]}[-50,1.2] $X_2$ \arrow(--x3){<->[][$*$]}[-130,1.2] 
 \schemestop
}
\caption{\subref{subfig:LT} The I1-FFL from~\cite{LanTu13}; \subref{subfig:ExA} Example A, reformatted.}
\end{figure}
\medskip

The system on the left is such an incoherent feed-forward loop, depicted using the traditional biochemical notation for positive and negative influence, while the system on the right is that of Example A of this paper, reformatted to resemble the system on the left.
The basis for the adaptive effect of the I1-FFL is the difference in time scale between the fast activation $X_1\to X_3$, which first leads to increase of $X_3$, and the slow inhibition $X_2\dashv X_3$, which reduces $X_3$ again on a longer time scale.

We can recognize the same working principle in the system of Example A on the right. An increase in input $X_1$ leads to an increase in both $X_2$ and `output' $X_3$; on the slower time scale of reaction $*$, the redistribution effect just described then reduces the value of $X_3$.

To conclude, the apparent discrepancy between the results of Lan, Tu, and co-authors on one hand and those of this paper can be traced back to a focus on different systems; the systems of this paper lie outside of the scope of~\cite{LanSartoriNeumannSourjikTu12,LanTu13}. If the systems of this paper are taken into account, then it is clear that `good adaptive performance', in the sense of high Precision and Sensitivity, can be achieved perfectly well in non-dissipative systems. 


%
%
%
%
%
\bibliography{ref}
\bibliographystyle{alphainitials}

\end{document}